\documentclass{article}
\usepackage[utf8]{inputenc}
\usepackage{xcolor}
\usepackage{geometry}
\usepackage{amsmath, amssymb, amsfonts, amsthm}
\usepackage{indentfirst}
\usepackage{cancel}
\usepackage{amssymb}
\usepackage{bm}
\usepackage{enumitem}
\usepackage{setspace}
\usepackage{hyperref}
\usepackage[all]{xy}
\usepackage{tikz}
\usepackage{diagbox}
\usepackage{array}
\usepackage{tabularx}
\usepackage{mathtools}
\usepackage{dsfont}
\usepackage{amsmath,calligra,mathrsfs}
\usepackage{adjustbox}
\usepackage{pgfplots}
\usepackage[textsize=tiny]{todonotes}
\usepackage[capitalize]{cleveref} % For some reason, this breaks when I put it after "\input tikz.tex" O_O

\usetikzlibrary{cd}
\newtheoremstyle{case}{}{}{}{}{}{:}{ }{}
\theoremstyle{case}
\newtheorem{case}{Case}
\usepackage{forloop}

\usepackage{tocloft}
\newcommand*{\sheafhom}{\mathscr{H}\kern -.5pt om}
\newcommand*{\sheafext}{\mathscr{E}\kern -.5pt xt}
\newcommand*{\sheafend}{\mathscr{E}\kern -.5pt nd}

\newcommand{\Syz}{\mathrm{Syz}}
\newcommand{\Ker}{\mathrm{Ker}}

\hypersetup{
	colorlinks=true,
	linkcolor=blue,
	linkbordercolor={0 0 1}}

\theoremstyle{plain}
\newtheorem{thm}{Theorem}[section] % reset theorem numbering for each chapter
\newtheorem{Lemma}[thm]{Lemma} % same for example numbers·
\newtheorem{Corollary}[thm]{Corollary} % same for example numbers
\newtheorem{Proposition}[thm]{Proposition} % same for example numbers

\theoremstyle{definition}
\newtheorem{construction}[thm]{Construction} % same for example numbers
\newtheorem{definition}[thm]{Definition} % definition numbers are dependent on theorem numbers
\newtheorem{term}[thm]{Terminology}
\newtheorem{conjecture}[thm]{Conjecture}
\newtheorem{remark}[thm]{Remark} % same for example numbers
\newtheorem{thm*}{Theorem*}

\title{Quantum Cohomology of a Fano Quiver Moduli Space}
\author{Junyu MENG}

\begin{document}
	\maketitle
	\begin{abstract}
		We consider a prime Fano 6-fold $Y$ of index 3, which is a fine quiver moduli space and a blow down of $\mathrm{Hilb}^3(\mathds{P}^2)$.
		We calculate the quantum cohomology ring of $Y$ and obtain Quantum Chevalley formulas for the Schubert type subvarieties. The famous Dubrovin's Conjecture relating the quantum cohomology and the derived category is verified for $Y$.  
	\end{abstract}

	%\tableofcontents
	
	\section{Introduction}
	%To study smooth projective varieties, a common philosophy is to associate to these varieties some algebraic structures. Usually, we can consider the cohomology, such as the singular cohomology ring, coherent sheaf cohomology and so on. By taking the cohomology, a lot of information could be lost in this procedure. As a result, we want to find out and concentrate on associated structures which are deeper and more enriched to keep more information. For example, the derived category of coherent sheaves is such a good associated structure: apart from the coherent sheaf cohomology information that we can recover easily, we have a triangulated category associated to each smooth projective variety.
	
	%Quantum cohomology ring is another good enriched associated structure: it is a deformation family of the ordinary cohomology ring. Apart from the usual information of intersections of cycles, the quantum cohomology ring also encodes how genus 0 curves intersect several cycles simultaneously inside the variety, i.e. Gromov-Witten invariants. Quantum cohomology ring attracts researchers' interests because it is an important ingredient in the mathematical formulation of Mirror Symmetry in terms of variations of Hodge structures. Moreover, quantum cohomology ring is an associative algebra, which means it encodes many inherent relations between Gromov-Witten invariants.
	
	For a smooth projective variety, there are two important objects associated to it: the derived category and the quantum cohomology. There is a conjectured link between the derived category and the quantum cohomology, guided by Homological Mirror Symmetry:
	\begin{conjecture}[Dubrovin, \cite{dubrovin}]
		Let $X$ be a smooth Fano variety. $\mathrm{D}^b(X)$ admits a full exceptional collection if and only if the big quantum cohomology ring of $X$ is generically semisimple.  
	\end{conjecture}
	There is also a refined version of this conjecture proposed by Kuznetsov and Smirnov \cite{Kuznetsov_Smirnov_2021}, formulated using the notion of Lefschetz collection.
	
	In this paper, we will try to check the above conjectures for a specific prehomogeneous variety $Y$, which is a prime Fano 6-fold of index 3. The variety $Y$ is constructed as the fine moduli space of stable representations of the 3-Kronecker quiver with dimension vector $\underline{d}=(2,3)$. And the bundles appearing in the universal representation over $Y$ can be identified with pullbacks of the universal subbundles over $\mathrm{Gr}(2,8),\mathrm{Gr}(3,6)$ via the embeddings of $Y$ into these Grassmannians. There are several interesting descriptions of our variety $Y$. It is isomorphic to the variety of trisecant planes of the Segre embedded $\mathds{P}(W)$ in the Segre embedding $\mathds{P}(W)\hookrightarrow\mathds{P}(S^2W)$. The variety $Y$ can be also viewed as a blow down of $\mathrm{Hilb}^3(\mathds{P}^2)$. Moreover $Y$ is isomorphic to the height-zero moduli space $M_{\mathds{P}^2}(4,-1,3)$ of stable sheaves on $\mathds{P}^2$ with $(r,c_1,c_2)=(4,-1,3)$.
	
	The variety $Y$ is endowed with the action of the 8-dimensional algebraic group $PGL(W)$, which makes it a prehomogeneous variety. We will construct an involution of $Y$ and prove that $\mathrm{Aut}(Y)\cong PGL(W)\rtimes\mathds{Z}_2$ in Proposition \ref{Aut} .
	
	In order to calculate the Gromov-Witten Invariants, the explicit descriptions
	of the varieties of lines and planes contained in $Y$ are given in Proposition \ref{F_1(Y)} and Proposition \ref{plane form} respectively. A rough investigation of the variety of parametrized conics is conducted in Proposition \ref{conic}.
	
	By picking a general 1-dimensional torus inside $PGL(W)$, we get the so-called Białynicki-Birula decomposition, which decomposes $Y$ into the disjoint union of affine cells, since the fixed point locus of a general 1-dimensional torus would be discrete points. As a result, $Y$ is a smooth compactification of $\mathds{A}_{\mathds{C}}^6$ and we calculate the singularity of the complement of $\mathds{A}_{\mathds{C}}^6$. Then we work out the fundamental classes of the closures of affine cells in the Chow ring of $Y$. A symmetry between some affine cells is observed and this is explained by a specific involution of $Y$.
	
	Using the Lemma of Graber (Lemma \ref{graber}), we endow the 3-pointed, genus 0, degree 1 Gromov-Witten Invariants with enumerativity meanings, which allow us to do explicit calculations. We give a presentation of the small quantum cohomology ring of $Y$ in terms of generators and relations in Theorem \ref{quantum thm}. These results are complemented by the calculation of a degree 2 Gromov-Witten Invariant in Lemma \ref{d=2GW} using essentially associativity relations, and we can deduce from it the Quantum Chevalley and Quantum Giambelli formulas.
	
	The small quantum cohomology ring of $Y$ will be checked to be generically semisimple, and thus the big quantum cohomology ring is also generically semisimple. This example is among the first several nontoric nonhomogenenous examples for which explicit presentations of quantum cohomology are known. Together with the study of $\mathrm{D}^b(Y)$ in \cite{Derquiver}, this implies:
	\begin{Proposition}
		Dubrovin's Conjecture holds for $Y$.
	\end{Proposition}
	For the refined conjecture in \cite{Kuznetsov_Smirnov_2021}, we give an $\mathrm{Aut}(Y)$-invariant full Lefschetz collection by doing one more mutation, which satisfies all the expected properties.
	
	\paragraph{Acknowledgements.}The author would like to thank his advisors Laurent MANIVEL and Thomas DEDIEU for leading him to this interesting topic and for their useful suggestions and revisions of the article. The author also wants to thank Pieter BELMANS for useful discussions and his help in drawing the diagram of eigenvalues of the hyperplane multiplication.

		\section{Geometry of Quiver Moduli Space $Y$}
		Let $W$ be a 3-dimensional vector space, we view an element $R$ in $\mathds{C}^2\otimes\mathds{C}^3\otimes W$ as a $2\times 3$ matrix with entries in $W$. We let $Y$ be the GIT quotient of $GL(2)\times GL(3)/\mathds{G}_m$ acting on $\mathds{P}(\mathds{C}^2\otimes\mathds{C}^3\otimes W)$, where the action of $(A,B)\in GL(2)\times GL(3)$ sends $R$ to $ARB^{-1}$ and $\mathds{G}_m$ is viewed as a subgroup of $GL(2)\times GL(3)$ via the diagonal embedding.
		
		\begin{Proposition}
			$Y$ is a smooth 6 dimensional Fano variety of Picard number 1 and index 3.
		\end{Proposition}
		
		\begin{construction}(First embedding)
			$Y$ is embedded into $\mathrm{Gr}(3,S^2W)$ as a closed subscheme as shown in \cite{EPS}. The map is given explicitly as follows. An element $[R]\in Y$ can be viewed as an orbit of $2\times 3 $ matrices with entries in $W$ such that the three maximal minors are linearly independent as elements in $S^2W$ (the last requirement is actually the GIT stability condition, as shown in \cite[Lemma 1]{EPS}). We consider the 3-dimensional subspace $H_R\subset S^2W$ generated by the three maximal minors and one can verify that $H_R$ doesn't depend on the choice of the $2\times 3$ matrix $R$ when $R$ varies in the orbit. The map from $Y$ to $\mathrm{Gr}(3,S^2W)$ sends $[R]\in Y$ to $H_R\subset S^2W$.
		\end{construction}
		\begin{Proposition}[\cite{IlievManivel}]
			$Y$ is the image of one of the two extremal contractions of $\mathrm{Hilb}^3(\mathds{P}(W^*))$ (Another extremal contraction is the Hilbert-Chow morphism). The extremal divisor is the closure of the subvariety of length 3 subschemes which consist of three collinear points.
		\end{Proposition}
		\begin{proof}
			By \cite{Hilb^3}, the extremal contraction  that is different from the Hilbert-Chow morphism can be constructed as follows: for a length 3 subscheme $Z\subset \mathds{P}(W^*)$, the restriction morphism $$\alpha_Z:\mathrm{H}^0(\mathds{P}(W^*),\mathcal{O}_{\mathds{P}(W^*)}(2))\rightarrow\mathrm{H}^0(Z,\mathcal{O}_{Z}\otimes\mathcal{O}_{\mathds{P}(W^*)}(2))\cong \mathds{C}^3$$ is surjective by \cite[Lemma 3.4]{Hilb^3}. Thus we can define a $SL(W)$-equivariant morphism
			$$\phi_1:\mathrm{Hilb}^3(\mathds{P}(W^*))\rightarrow\mathrm{Gr}(3,\mathrm{H}^0(\mathds{P}(W^*),\mathcal{O}_{\mathds{P}(W^*)}(2)))=\mathrm{Gr}(3,S^2W).$$ Consequently, the image of $\phi_1$ is the closure of the $SL(W)$-orbit of $\langle xy,yz,xz\rangle$, where $\{x,y,z\}$ is a basis for $W$. An easy calculation shows that the $SL(W)$-orbit of $\langle xy,yz,xz\rangle$ in $\mathrm{Gr}(3,S^2W)$ is 6-dimensional. Moreover, for $R=\begin{pmatrix}
				x&y&0\\
				0&y&z
			\end{pmatrix}$, $H_R$ is exactly $\langle xy,yz,xz\rangle$, which implies that this 6-dimensional orbit is contained in the embedded image of the 6-dimensional irreducible variety $Y$ inside $\mathrm{Gr}(3,S^2W)$. As a result, $Y$ is also the closure of the $SL(W)$-orbit of $\langle xy,yz,xz\rangle$.
			
			The last statement comes from the observation that if the embedding of a length 3 subscheme $Z$ into $\mathds{P}(W^*)$ can factor through $\{x=0\}\subset \mathds{P}(W^*)$, then the image of $Z$ via $\phi_1$ is equal to $\langle x^2,xy,xz\rangle$.
		\end{proof}
		We fix a basis $\{ x,y,z\}$ of $W$ from now on. We can analyze the orbits of the action of $SL(W)$ (or $PGL(W)$) on $Y$. 
		
		\begin{Corollary}
			There are five orbits and the Hasse diagram is as follows:
			
			
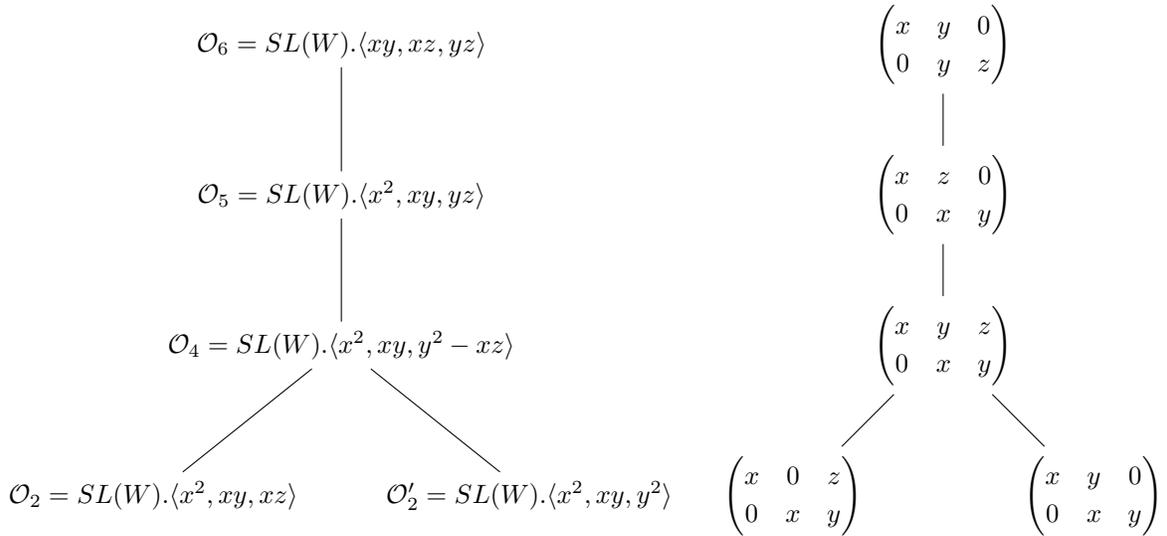
\begin{tikzpicture}
				\node (max) at (0,14) {$\mathcal{O}_6=SL(W).\langle xy,xz,yz\rangle$};
				\node (a) at (0,12) {$\mathcal{O}_5=SL(W).\langle x^2,xy,yz\rangle$};
				\node (b) at (0,10) {$\mathcal{O}_4=SL(W).\langle x^2,xy,y^2-xz\rangle$};
				\node (c) at (-2.5,8) {$\mathcal{O}_2=SL(W).\langle x^2,xy,xz\rangle$};
				\node (d) at (2.5,8) {$\mathcal{O}_2'=SL(W).\langle x^2,xy,y^2\rangle$};
				\draw (max) -- (a)--(b)--(c);
				\draw (b)--(d);
				\node (m) at (8,14) {$\begin{pmatrix}
						x & y& 0\\
						0&y &z
					\end{pmatrix}$};
				\node (a') at (8,12) {$\begin{pmatrix}
						x & z& 0\\
						0&x &y
					\end{pmatrix}$};
				\node (b') at (8,10) {$\begin{pmatrix}
						x & y& z\\
						0&x &y
					\end{pmatrix}$};
				\node (c') at (6,8) {$\begin{pmatrix}
						x & 0& z\\
						0&x &y
					\end{pmatrix}$};
				\node (d') at (10,8) {$\begin{pmatrix}
						x & y& 0\\
						0&x &y
					\end{pmatrix}$};
				\draw (m) -- (a')--(b')--(c');
				\draw (b')--(d');
			\end{tikzpicture}
			
			The dimensions of the orbits are equal to $6,5,4,2,2$ respectively from top to bottom. We will denote these orbits by $\mathcal{O}_6,\mathcal{O}_5,\mathcal{O}_4,\mathcal{O}_2,\mathcal{O}_2'$ as indicated in the diagram. In this Hasse diagram, a line connecting two orbits implies that the smaller dimensional orbit is contained in the closure of the bigger dimensional orbit.
		\end{Corollary}
		\begin{proof}
			We only need to analyze the $SL(W)$-orbit on $\mathrm{Hilb}^3(\mathds{P}(W^*))$ because the extremal contraction $\phi_1$ is obviously $SL(W)$-equivariant. The orbit of three points in general position in $\mathds{P}(W^*)$ corresponds to $\mathcal{O}_6$. The orbit of ``a point $p_1$ together with a tangent vector $v$ and another point $p_2$ not lying in the direction of the tangent vector $v$" corresponds to $\mathcal{O}_5$.
			
			A length 3 subscheme supported in one single point corresponds to $Spec(A)$ as an abstract scheme, where $A$ is a local Artinian $\mathds{C}$-algebra of length 3. As a small extension of $\mathds{C}[\epsilon]/\epsilon^2$, $A$ can only be either $\mathds{C}[\epsilon,\eta]/(\epsilon^2,\epsilon\eta,\eta^2)$ (the trivial extension) or $\mathds{C}[\epsilon]/\epsilon^3$ (the nontrivial extension). Up to automorphisms of $\mathds{P}(W^*)$, there is only one possible embedding of $Spec(\mathds{C}[\epsilon,\eta]/(\epsilon^2,\epsilon\eta,\eta^2))$ into $\mathds{P}(W^*)$, which corresponds to the orbit $\mathcal{O}_2'$. Up to automorphisms of $\mathds{P}(W^*)$, there are two possible embeddings of $Spec(\mathds{C}[\epsilon]/\epsilon^3)$ into $\mathds{P}(W^*)$. The first possible embedding corresponds to $\mathcal{O}_4$. Another possible embedding factors through some line $\mathds{P}^1\subset \mathds{P}(W^*)$ and the corresponding subscheme $Z$ is mapped via the extremal contraction $\phi_1$ into $\mathcal{O}_2$. The other $SL(W)$-orbits on $\mathrm{Hilb}^3(\mathds{P}(W^*))$ which are not mentioned above are all mapped to $\mathcal{O}_2$ via the extremal contraction $\phi_1$.
			
			The claims about the inclusion relations between the closures of orbits are easy to verify.
		\end{proof}
		\begin{Corollary}
			The blow up of $Y$ along the orbit $\mathcal{O}_2$ is isomorphic to $\mathrm{Hilb}^3(\mathds{P}(W^*))$.
		\end{Corollary}
		\begin{proof}
			We see from the above arguments that the complement of $\mathcal{O}_2$ is the biggest open locus where the extremal contraction $\phi_1:\mathrm{Hilb}^3(\mathds{P}(W^*))\rightarrow Y$ is an isomorphism. By construction, the fiber of $\phi_1$ over any point of $\mathcal{O}_2$ is isomorphic to  $\mathds{P}^3=\mathrm{Hilb}^3(\mathds{P}^1)$. Applying the universal property of $\mathrm{Bl}_{\mathcal{O}_2}Y$, we get a morphism from $\mathrm{Hilb}^3(\mathds{P}(W^*))$ to $\mathrm{Bl}_{\mathcal{O}_2}Y$, which must be an isomorphism by Zariski's Main Theorem.
		\end{proof}
		%The blow up of $Y$ along the $\mathcal{O}_2$ is exactly the Hilbert scheme $\mathrm{Hilb}^3(\mathds{P}^2)$ of three points on $\mathds{P}^2$ and the exceptional divisor in $\mathrm{Hilb}^3(\mathds{P}^2)$ is the locus where the 3 points on $\mathds{P}^2$ are collinear. The Hilbert scheme description is connected with the embedding $Y\subset \mathrm{Gr}(3,S^2W)$ as follows: except the case where $[R]\in Y$ lies in the orbit of elements in the form $\langle x^2,xy,xz\rangle$, $H_R$ always defines a length 3 subscheme on $\mathds{P}(W^*)$ when we view $H_R\subset S^2W$ as a subspace of $H^0(\mathds{P}(W^*),\mathcal{O}_{\mathds{P}(W^*)}(2))$. For example, when $H_R=\langle xy,yz,xz\rangle$ is viewed as a subspace of quadric forms, the subscheme it defines on $\mathds{P}(W^*)$ is the union of three reduced points $(x=y=0),(y=z=0),(x=z=0)$. 
		
		From the above proof, we see that $Y$ is the closure of elements in $\mathds{P}(\wedge^3(S^2W))$ of the form $ xy\wedge xz\wedge yz$. Hence $Y$ is isomorphic to the variety of trisecant planes of the Segre embedded $\mathds{P}(W)$ in the Segre embedding $\mathds{P}(W)\hookrightarrow\mathds{P}(S^2W)$ by \cite{IlievManivel}: The variety of trisecant planes of the Segre embedded $\mathds{P}(W)$ is the closure of elements in $\mathrm{Gr}(3,S^2W)\subset \mathds{P}(\wedge^3(S^2W))$ of the form $\langle x^2,y^2,z^2\rangle$ and what we claimed just follows from the fact that there is an automorphism of $\wedge^3(S^2W)$ sending elements of the form $e^2\wedge f^2\wedge g^2$ to $ef\wedge fg\wedge eg$.

		\begin{construction}[Second embedding]
			There is a natural embedding from $Y$ to $\mathrm{Gr}(2,S^{2,1}W)$. Recall that $S^{2,1}W$ can be viewed as the kernel of the map $W\otimes S^2W\rightarrow S^3W$. We define $\Syz(H_R)\subset S^{2,1}W$ to be the kernel of $H_R\otimes W\rightarrow S^3W$, one can actually show that $\Syz(H_R)$ is two dimensional except if $[R]$ is in $\mathcal{O}_2$. There is a canonical two-dimensional subspace $\mathrm{Syz}_R$ of $\Syz(H_R)$: if $R=\begin{pmatrix}
				A & B &C\\
				D & E& F
			\end{pmatrix}$  where $A,B,C,D,E,F\in W$, then we have two linearly independent syzygies $$A\otimes (BF-CE)-B\otimes(AF-CD)+C\otimes(AE-BD),$$
			$$D\otimes (BF-CE)-E\otimes(AF-CD)+F\otimes (AE-BD),$$
			both of which lie in the kernel of $W\otimes S^2W\rightarrow S^3W$, and hence can be viewed as elements of $S^{2,1}W$.
			And the two-dimensional subspace $\Syz_R$ generated by the two syzygies doesn't depend on the choice of $R$ in the orbit corresponding to $[R]$. We can thus define a morphism $Y\rightarrow \mathrm{Gr}(2,S^{2,1}W)$ by sending $[R]$ to $\Syz_R$.
		\end{construction}
		\begin{definition}
			We denote the pullback of the universal subbundle on $\mathrm{Gr}(3,S^2W)$ by $\mathcal{U}_2$. We denote the pullback of the universal subbundle on $\mathrm{Gr}(2,S^{2,1}W)$ by $\mathcal{U}_1$.
		\end{definition}
		\begin{construction}[Quiver moduli description and universal bundles]
			By the stability analysis conducted in Lemma 1 and Lemma 2 of \cite{EPS}, $Y$ can be viewed as the moduli space of stable representations of the 3-Kronecker quiver $\begin{tikzcd}
				\bullet \arrow[r,bend left] \arrow[r,bend right] \arrow[r] & \bullet
			\end{tikzcd}$, of dimension vector $\underline{d}=(2,3)$, with stability parameter $\theta=(3,-2)$. By choosing $\underline{a}=(-1,1)$ (see \cite{belmans2023rigidity} for the notation), we have a universal representation consisting of two universal bundles and a universal representation map $\mathcal{V}_1\rightarrow\mathcal{V}_2\otimes W$. It turns out that $\mathcal{V}_1=\mathcal{U}_1$ and $\mathcal{V}_2=\mathcal{U}_2$ (see \cite{ES}). By definition, we have the natural inclusion of vector bundles over $Y$: $\mathcal{U}_1\hookrightarrow S^{2,1}W$ and $\mathcal{U}_2\hookrightarrow S^2W$. The universal representation map $\mathcal{U}_1\rightarrow\mathcal{U}_2\otimes W$ is induced by the inclusion $S^{2,1}W\hookrightarrow S^2W\otimes W$. \label{univrep}
		\end{construction}

		\begin{Proposition}
			$c_1(\mathcal{U}_1^*)=c_1(\mathcal{U}_2^*)$ is the ample generator of the Picard group of $Y$.
		\end{Proposition}
		\begin{proof}
			See Section 4.
		\end{proof}
		
		\begin{remark}[Universal Property]\label{universal}
			The bundles $\mathcal{U}_i$ together with the morphisms $\mathcal{U}_1\rightarrow\mathcal{U}_2\otimes W$ satisfy the universal property up to a twist.
			More precisely, if there is a family of stable representations $\mathcal{V}_1\rightarrow\mathcal{V}_2\otimes V$ over a reduced scheme $S$, where $\mathcal{V}_1,\mathcal{V}_2$ are vector bundles over $S$ of rank 2 and 3 respectively, and $V$ is a 3-dimensional vector space together with a fixed isomorphism $V\cong W$, such that the fiber maps of $\mathcal{V}_1\rightarrow\mathcal{V}_2\otimes V$ are stable representations of the 3-Kronecker quiver, then there exists a unique map $f:S\rightarrow Y$ and a line bundle $\mathcal{L}$ over $S$, such that $f^*\mathcal{U}_i\cong\mathcal{V}_i\otimes\mathcal{L}$ and $\mathcal{V}_1\rightarrow\mathcal{V}_2\otimes V$ is exactly the pullback of $\mathcal{U}_1\rightarrow\mathcal{U}_2\otimes W$ (after twisting).
		\end{remark}
		\begin{construction}
			Moreover we recall from \cite{belmans2023chow} another description of $Y$: $Y$ is the zero locus of a general section of $Q^*(1)$ on $\mathrm{Gr}(2,S^{2,1}W)$, where $Q$ is the universal quotient bundle on $\mathrm{Gr}(2,S^{2,1}W)$. 
			
			We view $S^{2,1}W$ as an irreducible $SL(W)$ representation and it is actually isomorphic to $\mathfrak{sl}(W)$ with $SL(W)$ conjugation action (This follows from the fact that $\wedge^2W\otimes W\cong S^{2,1}W\oplus \wedge^3W$ and thus $S^{2,1}W=\Ker(\wedge^2W\otimes W\rightarrow \wedge^3W)$, together with the fact that $\wedge^2W\cong W^*\otimes \wedge^3W$, $\mathfrak{sl}(W)=\ker(W^*\otimes W\rightarrow \mathds{C})$). From now on we will identify $S^{2,1}W$ with $\mathfrak{sl}(W)$.
			
			From the tautological exact sequence $0\rightarrow Q^*\rightarrow S^{2,1}W^*\rightarrow S^*\rightarrow 0$, we have a filtration for $\wedge^3(S^{2,1}W^*)$: $0\subset F_1=\wedge^3Q^*\subset F_2\subset F_3=\wedge^3(S^{2,1}W^*)$ such that $F_2/F_1\cong \wedge^2 Q^*\otimes S^*$ and $\wedge^3(S^{2,1}W^*)/F_2\cong Q^*\otimes\wedge^2S^*\cong Q^*(1)$. By the Borel-Bott-Weil Theorem we know that $H^0(\mathrm{Gr}(2,S^{2,1}W), Q^*(1))\cong \wedge^3(S^{2,1}W^*)$. Let us consider the following $SL(W)$-invariant alternating 3-form $\omega$ on $\mathfrak{sl}(W)$: $(A,B,C)\in \mathfrak{sl}(W)^{\oplus 3}\mapsto \mathrm{Tr}(A[B,C])$. The stabiliser group in $GL(S^{2,1}W)$ of this 3-form $\omega$ is 8-dimensional, so the orbit of this 3-form is the open orbit in $\wedge^3(\mathfrak{sl}(W))^*$. Then the image of $\omega$ under the vector bundle map $\wedge^3(\mathfrak{sl}(W))^*\rightarrow Q^*(1)=\wedge^3(\mathfrak{sl}(W))^*/F_2$ is a general section of $Q^*(1)$ and it vanishes at a point $\langle A,B\rangle\in \mathrm{Gr}(2, \mathfrak{sl}(W))$ iff $\mathrm{Tr}(C[A,B])=0$ for any $C\in \mathfrak{sl}(W)$, that is iff $AB=BA$. In view of this, the zero locus of this global section can be viewed as the variety of abelian planes in $\mathfrak{sl}(W)$, i.e. the variety of those planes whose elements commute with each other as matrices. 
			
			One can show that this zero locus is exactly $Y$ being embedded in $\mathrm{Gr}(2, S^{2,1}W)$ in the way mentioned above: We consider the representative $\langle xy,xz,yz\rangle\in Y\subset \mathrm{Gr}(3,S^2W)$ of the open orbit. It is mapped to $\langle E_{22}-E_{11},E_{33}-E_{22}\rangle\in \mathrm{Gr}(2,\mathfrak{sl}(W))$ via the embedding $Y\hookrightarrow \mathrm{Gr}(2,\mathfrak{sl}(W))$, which is an abelian plane contained in $\mathfrak{sl}_3$ and we can thus conclude that Y is contained in the zero locus of the global section. Since the zero locus of such a general global section has the expected dimension 6, some calculations of the classical invariants as in \cite{belmans2023chow} will show that $Y$ is the only irreducible component of the zero locus.
			\begin{Proposition}
				$Y$ is the variety of abelian planes in $\mathfrak{sl}(W)$, i.e. the subvariety of $\mathrm{Gr}(2,\mathfrak{sl}(W))$ consisting of those planes whose elements commute with each other as matrices.
			\end{Proposition}
		\end{construction}
		
		A general abelian plane contained in $\mathfrak{sl}(W)$ is actually of the form $\langle\mathrm{diag}(1,1,-2),\mathrm{diag}(1,-2,1)\rangle$. More precisely, if we view elements in $\mathfrak{sl}(W)$ as endomorphisms of $W^*$, then a general abelian plane in $\mathfrak{sl}(W)$ is in the above diagonal matrix form once we choose a suitable basis for $W^*$. The suitable basis we choose for the three dimensional space $W^*$ gives three distinct points on $\mathds{P}(W^*)$ and this gives us the link between the embedding $Y\subset \mathrm{Gr}(2,\mathfrak{sl}(W))$ and the description related to $\mathrm{Hilb}^3(\mathds{P}(W^*))$. 
		
		Finally, we mention that $Y$ is isomorphic to the height-zero moduli space $M_{\mathds{P}^2}(4,-1,3)$ of stable sheaves on $\mathds{P}^2$ with $(r,c_1,c_2)=(4,-1,3)$ by \cite{drezet}.
		\subsection{Automorphism Group}\label{automorphismsection}
		In this subsection, we deal with the automorphism group of $Y$. First of all, we have the 8-dimensional algebraic group $PGL(W)$ acting on $Y$. 
		\begin{Proposition}
			$PGL(W)=\mathrm{Aut}^{\circ}(Y)$.
		\end{Proposition}
		\begin{proof}
			We only need to show that $\dim(\mathrm{Aut}(Y))=8$. To calculate the dimension of $\mathrm{Aut}(Y)$, We use a quiver representation theoretical exact sequence involving the tangent bundle of $Y$ (see \cite{belmans2023chow}):
			$$0\rightarrow\mathcal{O}_Y\rightarrow \mathcal{E}nd(\mathcal{U}_1)\oplus\mathcal{E}nd(\mathcal{U}_2)\rightarrow \mathcal{H}om(\mathcal{U}_1,\mathcal{U}_2)^{\oplus 3}\rightarrow T_Y\rightarrow 0$$
			By \cite{belmans2023vector}, all the higher cohomology of $\mathcal{O}_Y,\mathcal{E}nd(\mathcal{U}_1),\mathcal{E}nd(\mathcal{U}_2),\mathcal{H}om(\mathcal{U}_1,\mathcal{U}_2)$ vanish and the dimensions of spaces of global sections are equal to $1,1,1,3$ respectively.  We get from the above exact sequence that $\dim(H^0(Y,T_Y))=8$, which allows us to conclude that $PGL(W)=\mathrm{Aut}^{\circ}(Y)$.
		\end{proof}   
		
		Using the description of $Y$ as the variety of abelian planes contained in $\mathfrak{sl}(W)$, we can define an outer automorphism $g_0$ of $Y$: For an abelian plane $\langle A,B\rangle\subset \mathfrak{sl}_3$, $f(\langle A,B\rangle)=\langle A^t, B^t\rangle$, where $A^t$ is the transpose matrix. To give a geometric interpretation of $g_0$, we investigate another idea of constructing automorphisms below.
		
		Recall that $Y$ is a blow down of the Hilbert scheme of 3 points on $\mathds{P}(W^*)$. We consider the following birational map $f$ from $Y$ to $Y^*$, where $Y^*$ is defined to be a blow down of the Hilbert scheme of 3 points on $\mathds{P}((W^*)^*)$ (in the parallel construction for $Y$):
		the map $f$ is defined on the open orbit $\mathcal{O}_6$ of elements with geometric support being 3 distinct points in general position on $\mathds{P}(W^*)$. Obviously 3 distinct points in general position on $\mathds{P}(W^*)$ correspond to 3 nonconcurrent lines joining 2 of the 3 points on $\mathds{P}(W^*)$ and thus correspond to 3 distinct points in general position on $\mathds{P}(W)$ and the action of $f$ is to send the 3 distinct points in general position on $\mathds{P}(W^*)$ to the 3 distinct points in general position on $\mathds{P}(W)$. 
		
		Now we consider the nondegenerate bilinear form on $W$ which is in the form $\begin{pmatrix}
			1&0&0\\
			0&1&0\\
			0&0&1
		\end{pmatrix}$ when written as a matrix with respect to the basis $\{x,y,z\}$. This gives us an isomorphism $W\cong W^*$, and thus an isomorphism $Y^*\cong Y$. If we compose this isomorphism with $f:Y\rightarrow Y^*$, we get a birational map $g_1$ on $Y$.
		
		\begin{Proposition}
			The involution morphism $g_0$ is actually the extension of the birational map $g_1$. 
		\end{Proposition}
		\begin{proof}
			The images of $\langle xy,xz,yz\rangle$ (where $x,y,z$ is the orthonormal basis above) via the maps $g_0, g_1$ are identical and we need to check that this holds for all $A.\langle xy,xz,yz\rangle$ where $A\in GL(W)$, because such elements constitute the open dense $GL(W)$-orbit $\mathcal{O}_6$. By the famous QR-decomposition, any such $A$ can be written as $A=QR$ where $Q$ is an orthogonal matrix and $R$ is an upper triangular matrix. Because orthogonal matrices preserve the above nondegenerate bilinear form, we only need to check for the case $A=R$ is upper triangular, and brutal calculation gives us the positive answer.
		\end{proof} 
		\begin{Proposition}
			$\mathrm{Aut}(Y)\cong PGL(W)\rtimes\mathds{Z}_2$, where $\mathds{Z}_2$ is embedded in $\mathrm{Aut}(Y)$ as the subgroup $\langle\mathrm{Id}_Y,g_0\rangle$.\label{Aut}
		\end{Proposition}
		\begin{proof}
			We consider an arbitrary element $f\in \mathrm{Aut}(Y)$. Since $PGL(W)=\mathrm{Aut}^{\circ}(Y)$, we know that $PGL(W)$ is a normal subgroup. Consequently, the image $f(\mathcal{O}_2)$ of the 2-dimensional $PGL(W)$-orbit $\mathcal{O}_2$ must be also $PGL(W)$-invariant and thus equal to either $\mathcal{O}_2$ or $\mathcal{O}_2'$.
			
			One can check that the involution $g_0$ would interchange $\mathcal{O}_2$ with $\mathcal{O}_2'$. After composing with $g_0$ if necessary, we can assume that $f(\mathcal{O}_2)=\mathcal{O}_2$. Then $f$ must induce an automorphism of $\mathrm{Bl}_{\mathcal{O}_2}Y\cong \mathrm{Hilb}^3(\mathds{P}(W^*))$. By \cite{autHilb}, $\mathrm{Aut}(\mathrm{Hilb}^3(\mathds{P}(W^*)))\cong PGL(W)$, which finishes the proof.
		\end{proof}
		\section{Interesting Subvarieties}
		\subsection{Lines, Planes and Conics}
		First we consider planes contained in $Y\subset \mathrm{Gr}(2,\mathfrak{sl}(W))$. If $\Pi$ is a plane contained in $Y$, we pick a smooth conic contained in $\Pi\subset Y$ and thus obtain a degree 2 morphism from $\mathds{P}^1$ to $Y$. Recall that $\det(\mathcal{U}_1^*)=\mathcal{O}_Y(1)$. We consider the pullback of the bundle $\mathcal{U}_1^*$ to $\mathds{P}^1$, which must split either as $\mathcal{O}_{\mathds{P}^1}(2)\oplus \mathcal{O}_{\mathds{P}^1} $ or as $\mathcal{O}_{\mathds{P}^1}(1)\oplus \mathcal{O}_{\mathds{P}^1}(1)$ because $\mathcal{U}_1^*$ is globally generated. An easy discussion shows that after possible reparametrizations, any degree two map $\mathds{P}^1\rightarrow \mathrm{Gr}(2,\mathfrak{sl}(W))$ is in one of the following forms: 
		\begin{equation*}
			\begin{split}
				& \bullet \ (s:t)\mapsto \langle se_1+te_2, se_3+te_4\rangle, \ \mathrm{e_i\in \mathfrak{sl}(W)  \ are \ linearly \ independent}\\
				&\bullet \ (s:t)\mapsto \langle se_1+te_2, se_3+te_1\rangle, \ \mathrm{e_i\in \mathfrak{sl}(W) \ are  \ linearly \ independent} \\
				&\bullet \ (s:t)\mapsto \langle e_1, s^2e_2+ste_3+t^2e_4\rangle,  \ \mathrm{0\neq e_1\in \mathfrak{sl}(W) } 
			\end{split}
		\end{equation*}
		The three types of conics are called $\tau,\rho,\sigma$-conics respectively. 
		
		The 2-plane linearly generated by the conic is exactly the original 2-plane $\Pi$ that we are considering: $\Pi\subset \mathds{P}(\wedge^2(\mathfrak{sl}(W)))$ is equal to $\langle e_1\wedge e_3,e_1\wedge e_4+e_2\wedge e_3,e_2\wedge e_4\rangle$, $\langle e_1\wedge e_3,e_2\wedge e_3,e_2\wedge e_1\rangle$, $\langle e_1\wedge e_2,e_1\wedge e_3,e_1\wedge e_4\rangle$ for the three types of conics respectively. The case $\Pi=\langle e_1\wedge e_3,e_1\wedge e_4+e_2\wedge e_3,e_2\wedge e_4\rangle$ is easily excluded because such $\Pi\subset \mathds{P}(\wedge^2(\mathfrak{sl}(W)))$ isn't contained in $\mathrm{Gr}(2,\mathfrak{sl}(W))$. The cases $\Pi=\langle e_1\wedge e_3,e_2\wedge e_3,e_2\wedge e_1\rangle$ and $\Pi=\langle e_1\wedge e_2,e_1\wedge e_3,e_1\wedge e_4\rangle$ are called $\rho$-plane and $\sigma$-plane respectively. 
		\begin{Proposition}
			The variety of planes contained in $Y$ is isomorphic to $\mathds{P}(W)\times \mathds{P}(W^*)$.\label{plane form}
		\end{Proposition}
		\begin{proof}
			Recall that $Y$ is described as the variety of abelian planes contained in $\mathrm{Gr}(2,\mathfrak{sl}(W))$. A $\sigma$-plane contained in $\mathrm{Gr}(2,\mathfrak{sl}(W))$ is determined by the data $V_1\subset V_4\subset \mathfrak{sl}(W)$, where $V_1,V_4$ are vector subspaces of dimension 1 and 4 respectively (actually $V_1=\mathds{C}.e_1,V_4=\mathds{C}\langle e_1,e_2,e_3,e_4\rangle$). In order for such a $\sigma$-plane to be contained in $Y\subset \mathrm{Gr}(2,\mathfrak{sl}(W))$, we need to ensure that every matrix element in $V_4$ commutes with $V_1$. 
			
			For a nonzero matrix $M\in \mathfrak{sl}(W)$, the space of matrices in $\mathfrak{sl}(W)$ commuting with $M$ is either two dimensional or four dimensional, and the latter 4-dimension case holds exactly when $M$ is similar to either $\begin{pmatrix}
				1&0&0\\
				0&1&0\\
				0&0&-2
			\end{pmatrix}$ or $\begin{pmatrix}
				0&1&0\\
				0&0&0\\
				0&0&0
			\end{pmatrix}$. Another interpretation of the above condition is that $M-\lambda.\mathrm{Id}$ is of rank 1 for some $\lambda\in \mathds{C}$. Thus the subvariety of such $M$ is described in $\mathds{P}(\mathfrak{sl}(W))$ as the image of the embedding $\mathds{P}(W)\times \mathds{P}(W^*)\hookrightarrow\mathds{P}(W\otimes W^*)\dashrightarrow\mathds{P}(\mathfrak{sl}(W))$, where the first map is the Segre embedding and the second map is the linear projection from $[\mathds{C}.\mathrm{Id}_W]\in \mathds{P}(W\otimes W^*)$.
			
			As a result of the linear algebra above, in order for the $\sigma$-plane given by $V_1\subset V_4$ to be contained in $Y$, it is necessary that $V_1$ represents a point in $\mathds{P}(W)\times \mathds{P}(W^*)\subset\mathds{P}(\mathfrak{sl}(W))$ and $V_4$ will be completely determined by $V_1$.
			
			It remains to show that there is no $\rho$-plane contained in $Y$. Suppose that a $\rho$-plane $\Pi=\langle e_1\wedge e_3,e_2\wedge e_3,e_2\wedge e_1\rangle$ is contained in $Y$, then $e_1,e_2,e_3$ must commute with each other as matrices in $\mathfrak{sl}(W)$, so that $\langle e_1,e_2,e_3\rangle$ is an abelian 3-plane contained in $\mathfrak{sl}(W)$. A priori, any matrix in this abelian 3-plane must be similar to either $\begin{pmatrix}
				1&0&0\\
				0&1&0\\
				0&0&-2
			\end{pmatrix}$ or $\begin{pmatrix}
				0&1&0\\
				0&0&0\\
				0&0&0
			\end{pmatrix}$ by the linear algebra above. One can show that there must be some elements in this abelian 3-plane similar to $\begin{pmatrix}
				1&0&0\\
				0&1&0\\
				0&0&-2
			\end{pmatrix}$ and get a contradiction after some calculations, i.e. there cannot be any abelian 3-plane contained in $\mathfrak{sl}(W)$.
		\end{proof}
		
		\begin{remark}\label{planes}
			Notice that when a nonzero matrix in $V_1$ is similar to $\begin{pmatrix}
				1&0&0\\
				0&1&0\\
				0&0&-2
			\end{pmatrix}$, the plane $[V_1\subset V_4]$ in $Y$ has nonempty intersection with the open orbit in $Y$. In geometrical terms (viewing $Y$ as a blow down of $\mathrm{Hilb}^3(\mathds{P}^2)$), this plane in $Y$ corresponds to the image of the following subvariety of $\mathrm{Hilb}^3(\mathds{P}^2)$: pick a line $l$ in $\mathds{P}^2$ and a point $p_1$ not contained in $l$ and consider the subvariety $\{(p_1,(p_2,p_3))\in \mathrm{Hilb}^3(\mathds{P}^2)|(p_2,p_3)\in \mathrm{Hilb}^2(l)\}$. Via the embedding into $\mathrm{Gr}(3,S^2W)$, such kind of plane is given for example by $(a:b:c)\mapsto xz\wedge yz\wedge (ax^2+bxy+cy^2)$ and $p_1=(x=y=0)$ and $l=(z=0)$. 
			
			However, when a nonzero matrix in $V_1$ is similar to $\begin{pmatrix}
				0&1&0\\
				0&0&0\\
				0&0&0
			\end{pmatrix}$, the plane $[V_1\subset V_4]$ in $Y$ has empty intersection with the open orbit in $Y$.
		\end{remark}
		
		\begin{Corollary}
			Through a general point $p\in Y$, there pass 3 planes contained in $Y$.\label{3 planes}
		\end{Corollary}
		\begin{proof}
			A general point $p\in Y$ is described geometrically (viewing $Y$ as a blow down of $\mathrm{Hilb}^3(\mathds{P}^2)$) as the 0-dimensional subscheme of 3 points in general position on $\mathds{P}^2$. Then this Corollary follows from the previous Remark: we pick one of the three points to be $p_1$ and pick $l$ to be the line connecting the other two points.
		\end{proof}
		\begin{Proposition}
			The Fano variety of lines $F_1(Y)$ is isomorphic to the projective bundle $\mathds{P}((N')^*)=\mathrm{Proj}(\mathrm{Sym}^.(N'))$ over $\mathds{P}(W)\times \mathds{P}(W^*)$, where $N'$ is the normal bundle of $\mathds{P}(W)\times \mathds{P}(W^*)$ in $\mathds{P}(\mathfrak{sl}(W))$.\label{F_1(Y)}
		\end{Proposition}
		
		%On the other hand, we have the exact sequence $0\rightarrow \mathcal{O}_{\mathds{P}^7}(1)|_Z=\mathcal{O}_Z(1,1)\rightarrow N_{Z|\mathds{P}^8}\rightarrow N_{Z|\mathds{P}^7}\rightarrow 0$. To calculate $N_{Z|\mathds{P}^8}$, observe that we have the following exact sequence: $0\rightarrow T_{\mathds{P}(W)}\oplus T_{\mathds{P}(W^*)} \rightarrow T_{\mathds{P}(W\otimes W^*)}\rightarrow T_{\mathds{P}(W)}\boxtimes T_{\mathds{P}(W^*)}\rightarrow 0$, so that $N_{Z|\mathds{P}^8}\cong T_{\mathds{P}(W)}\boxtimes T_{\mathds{P}(W^*)}$.
		\begin{proof}
			Via the embedding of $Y$ into $\mathrm{Gr}(2,\mathfrak{sl}(W))$, we know that a line in $Y$ is determined by the data $V_1\subset V_3\subset \mathfrak{sl}(W)$ such that elements in $V_3$ commute with $V_1$. Using the argument in the proof of the last Proposition, we know that $V_1$ represents a point in $\mathds{P}(W)\times \mathds{P}(W^*)\subset\mathds{P}(\mathfrak{sl}(W))$ and $V_3$ is contained in the four dimensional subspace $V_4$ determined by $V_1$. Thus every line in $Y$ is contained in a unique plane in $Y$ and $F_1(Y)$ is a $\mathds{P}^2$-bundle over $\mathds{P}(W)\times \mathds{P}(W^*)$. 
			
			To determine which $\mathds{P}^2$-bundle it is, we first determine the vector bundle of trace zero matrices commuting with the element $[v\otimes w^*]\in \mathds{P}(W)\times \mathds{P}(W^*)\subset\mathds{P}(W\otimes W^*)$. A rank one matrix $a\otimes b^*$ commutes with $v\otimes w^*$ iff either $[v\otimes w^*]=[a\otimes b^*]\in \mathds{P}(W)\times \mathds{P}(W^*)$ or $\langle v,b^*\rangle=\langle a,w^*\rangle=0$. The latter condition is equivalent to saying that $a\otimes b^*$ lies in the fiber of the image of the bundle map $Q^*\boxtimes Q'^*\rightarrow W^*\otimes W=W\otimes W^*$ at the point $[v\otimes w^*]$, where $Q,Q'$ are the universal quotient bundles on $\mathds{P}(W),\mathds{P}(W^*)$ respectively. And all the matrices in $W\otimes W^*$ commuting with $v\otimes w^*$ lie in the five dimensional subspace generated by $\mathrm{Id}_W$ and $a\otimes b^*$ satisfying $\langle v,b^*\rangle=\langle a,w^*\rangle=0$. From these observations, we conclude that the vector bundle over $\mathds{P}(W)\times \mathds{P}(W^*)$ of trace zero matrices commuting with element $[v\otimes w^*]\in \mathds{P}(W)\times \mathds{P}(W^*)$ is exactly the image bundle of $Q^*\boxtimes Q'^*$ under the map $Q^*\boxtimes Q'^*\rightarrow W^*\otimes W\rightarrow\mathfrak{sl}(W)$, which is still isomorphic to $Q^*\boxtimes Q'^*$ itself.
			
			Now that the bundle of varying $V_4$ is just $Q^*\boxtimes Q'^*$, we need to consider the bundle of varying $V_4/V_1$. The image of the map $\mathcal{O}(-1,-1)\rightarrow W\otimes W^*\rightarrow \mathfrak{sl}(W)$ is contained in the image bundle $Q^*\boxtimes Q'^*$ mentioned above. The quotient bundle $Q^*\boxtimes Q'^*/\mathcal{O}(-1,-1)$ is exactly the bundle of varying $V_4/V_1$. As a result, $F_1(Y)\cong \mathds{P}((Q^*\boxtimes Q'^*/\mathcal{O}(-1,-1))^*)$.
			
			To finish the proof, we only need to show that $Q^*\boxtimes Q'^*/\mathcal{O}(-1,-1)\cong N'\otimes \mathcal{O}(-2,-2)$. We have the natural exact sequence $0\rightarrow \mathcal{O}(1,1)\rightarrow N\rightarrow N'\rightarrow 0$, where $N$ is the normal bundle of Segre embedding $\mathds{P}(W)\times \mathds{P}(W^*)\subset\mathds{P}(W\otimes W^*)$. To calculate $N$, observe that we have the following exact sequence: $0\rightarrow T_{\mathds{P}(W)}\oplus T_{\mathds{P}(W^*)} \rightarrow T_{\mathds{P}(W\otimes W^*)}\rightarrow T_{\mathds{P}(W)}\boxtimes T_{\mathds{P}(W^*)}\rightarrow 0$, which shows that $N\cong T_{\mathds{P}(W)}\boxtimes T_{\mathds{P}(W^*)}\cong (Q\boxtimes Q')\otimes\mathcal{O}(1,1)\cong (Q^*\boxtimes Q'^*)\otimes\mathcal{O}(2,2) $.
			%In particular, $F_1(Y)$ is connected in the complex topology and thus irreducible in the Zariski topology. Moreover, when a nonzero matrix in $V_1$ is similar to $\begin{pmatrix} 1&0&0\\0&1&0\\0&0&-2\end{pmatrix}$, every line $[V_1\subset V_3]$ has nonempty intersection with the open orbit in $Y$ and we can see that it is contained in exactly one plane in $Y$ by reading out what $p_1$ and $l$ is in the Remark \ref{planes}. Consequently we have an open subset of $F_1(Y)$ isomorphic to a $\mathds{P}^2$ bundle over $\{(p_1,l)\in \mathds{P}(W^*)\times \mathds{P}(W)|p_1\notin l\}$.
		\end{proof}
		\begin{Corollary}
			Every line in $Y$ is contained in a unique plane in $Y$.\label{lineplane}
		\end{Corollary}
		\begin{term}\label{terminology}
			We say a plane $\Pi$ contained in $Y$ is general if it is given by an element in $\{(p_1,l)\in \mathds{P}(W^*)\times \mathds{P}(W)| \ p_1\notin l\}$ as in Remark \ref{planes}. A line $l$ contained in $Y$ is general if it is contained in a general plane $\Pi\subset Y$. A line or a plane in $Y$ being general is actually equivalent to saying that this line or plane intersects the biggest open orbit of $Y$.
		\end{term}
		
		Now we come to discuss conics contained in $Y$. Since our main focus is the Kontsevich moduli space, we will concentrate on ``parametrized conics", i.e. degree 2 morphisms from $\mathds{P}^1$ to $Y$. We have discussed the differences between $\tau,\rho,\sigma$-conics before. In any case, such morphism $\mathds{P}^1\rightarrow \mathrm{Gr}(2,\mathfrak{sl}(W))$ factors through some closed subvariety $\mathrm{Gr}(2,W_4)\subset \mathrm{Gr}(2,\mathfrak{sl}(W))$, where $W_4$ is a 4-dimensional subspace of $\mathfrak{sl}(W)$. 
		
		(Recall that $Y$ is the locus in $\mathrm{Gr}(2,\mathfrak{sl}(W))$ where a general 3-form $\omega$ on $\mathfrak{sl}(W)$ vanishes.)
		
		\begin{case}Assume the restriction of the general 3-form $\omega$ to $W_4$ is nonzero, we fix an isomorphism $\wedge^4W_4\cong \mathds{C}$. Then there exist a nonzero vector $v\in W_4$, such that $\omega(x,y,z)=x\wedge y\wedge z\wedge v\in\mathds{C}$, for any $x,y,z\in W_4$. A 2-plane $\langle x,y\rangle$ represents a point in $Y\cap \mathrm{Gr}(2,W_4)$ iff  $\omega(x,y,z)=0$ for any $z\in \mathfrak{sl}(W)$. In particular, when $z$ varies in $W_4$, this requires that $x\wedge y\wedge v=0$, i.e. $v\in \langle x,y\rangle$. As a result, if a conic is contained in this $Y\cap \mathrm{Gr}(2,W_4)$, then this conic must be a $\sigma$-conic, where $v$ is equal to the $e_1$ in the definition of $\sigma$-conics. Notice that every $\sigma$-conic is contained in a $\sigma$-plane in $Y$. As a result, the dimension of the subvariety of such $\sigma$-conics is 9+3=12 (without identifying isomorphic morphisms from $\mathds{P}^1$ to $Y$).
		\end{case}
		
		\begin{case}Assume the restriction of the general 3-form $\omega$ to $W_4$ is zero, which is equivalent to saying that $W_4\in \mathrm{Gr}(4,\mathfrak{sl}(W))$ lies in the zero locus of a general global section of $\wedge^3\mathcal{U}_4^*$, where $\mathcal{U}_4$ is the rank 4 universal subbundle over $\mathrm{Gr}(4,\mathfrak{sl}(W))$. The Pl\"{u}cker embedding embeds $\mathrm{Gr}(2,W_4)$ as a quadric hypersurface in $\mathds{P}(\wedge^2W_4)$. Any conic is contained in some 2-plane $\mathds{P}(W_3)\subset\mathds{P}(\wedge^2W_4)$ and thus must be equal to $\mathds{P}(W_3)\cap \mathrm{Gr}(2,W_4)$, unless $\mathds{P}(W_3)\subset \mathrm{Gr}(2,W_4)$ which is not the general case. In order for the conic to be contained in $Y$, we need that $W_3$ is contained in the kernel of the map $\wedge^2W_4\rightarrow \mathfrak{sl}(W)/W_4$ induced by the 3-form $\omega$. In particular, $W_4$ must satisfy $\dim(\mathrm{Ker}(\wedge^2W_4\rightarrow \mathfrak{sl}(W)/W_4))\geq 3$. Now we have in total 7 conditions on $[W_4]\in \mathrm{Gr}(4,\mathfrak{sl}(W))$ and thus a (9+3)-dimensional subscheme of the variety of parametrized conics. This subscheme is nonempty, for example we have the conic  $(s:t)\mapsto \langle \begin{pmatrix}
				0&-s&0\\
				t&0&-s\\
				0&0&0
			\end{pmatrix}, \begin{pmatrix}
				t&0&-3s\\
				0&t&0\\
				0&0&-2t
			\end{pmatrix}\rangle\in\mathrm{Gr}(2,\mathfrak{sl}(W))$.
		\end{case}
		\begin{Proposition}
			The Kontsevich moduli space of 3-pointed stable maps of degree 2 from genus 0 curves to $Y$ has at least two irreducible components of dimension 12.\label{conic}
		\end{Proposition}
		\subsection{Restrictions of Schubert Subvarieties}
		For future convenience, we also include some descriptions of the restrictions of Schubert subvarieties: For the embedding $Y\subset\mathrm{Gr}(2,\mathfrak{sl}(W))$, the Schubert cycle $c_2(\mathcal{U}_1^*)$ is defined by taking a general linear form on $\mathfrak{sl}(W)$, or equivalently picking a general 7-dimensional subspace $U_7\subset\mathfrak{sl}(W)$ and define the cycle $c_2(\mathcal{U}_1^*)$ to be the subvariety of abelian planes contained in $U_7\subset\mathfrak{sl}(W)$.
		
		For the embedding $Y\subset\mathrm{Gr}(3,S^2W)$, the Schubert cycle $c_3(\mathcal{U}_2^*)$ is defined by taking a general linear form $f$ on $S^2W$ and define the cycle $c_3(\mathcal{U}_2^*)$ to be the subvariety of 3-dimensional subspace contained in $\Ker(f)$ (and a priori the 3-dimensional subspace need to represent a point inside $Y$).
		
		Finally for the Schubert cycle $c_2(\mathcal{U}_2^*)$, we pick two general linear forms $q_1,q_2\in S^2W^*$. A point $\langle A,B,C\rangle\in Y$ (where $A,B,C\in S^2W$ are linearly independent) lies in this Schubert cycle if and only if
		$$\mathrm{rank}\begin{pmatrix}
			q_1(A)&q_1(B)&q_1(C)\\
			q_2(A)&q_2(B)&q_2(C)
		\end{pmatrix}\leq 1.$$

		\subsection{Białynicki-Birula decomposition}
		We apply the famous Białynicki-Birula decomposition Theorem to get a cellular decomposition of $Y$. 
		\begin{thm}[Białynicki-Birula decomposition, \cite{BBdecomposition}]
			Assume $X$ is a smooth projective algebraic variety equipped with the action of $G=\mathds{G}_m$. Let $X^G=\bigcup_iX_i$ be the decomposition of the fixed point set into (disjoint) irreducible components. Assume that $T_X|_{X_i}=T^+(X_i)\oplus T^-(X_i)\oplus T_{X_i}$ is the decomposition of $T_X|_{X_i}$ into the part where $\mathds{G}_{m}$ acts with positive, negative and zero weights respectively. For each component $X_i$, we define a subvariety $C_i=\{x\in X|\lim_{t\rightarrow\infty}t.x\in X_i\}$.
			
			Then $C_i$ is actually a locally closed subvariety and $X=\bigsqcup_iC_i$. The map $C_i\rightarrow X_i$ sending $x\in C_i$ to $\lim_{t\rightarrow\infty}t.x$ is a regular morphism, making $C_i$ a fibration over $X_i$. Moreover, this fibration is actually isomorphic to the vector bundle $T^-(X_i)$.
		\end{thm}
		By choosing a general one dimensional subtorus $T$ of $GL(W)$, we can ensure that the action of $T$ on $Y$ has only finitely many fixed points: let $\{x,y,z\}$ be a basis of $W$, we choose the embedding of $\mathds{G}_m$ into $GL(W)$ to be $t\rightarrow \begin{pmatrix}
			1 & 0&0\\
			0& t &0\\
			0 &0 &t^3
		\end{pmatrix}.$ 
		By calculation, the action of this torus has 13 fixed points $R_1,\dots, R_{13}$ (listed in the Hasse diagram below). We can associate to each fixed point $R_i$ the cell $C_{R_i}=\{R\in Y| \lim_{t\rightarrow\infty}t.R=R_i\}$. By the Białynicki-Birula decomposition Theorem, these disjoint cells are isomorphic to $\mathds{A}^{n_i}$ as locally closed subvarieties of $Y$, and the closures $\overline{C_{R_i}}$ freely generate the singular homology group of $Y$ by \cite[Theorem 4.4, Part II by James B.Carrell]{BBbook}. From the Białynicki-Birula decomposition and the exact sequence $A^k(Z)\rightarrow A^k(X)\rightarrow A^k(X-Z)\rightarrow 0$ for any closed subscheme $Z$ of an algebraic scheme $X$, we know that the Chow group of $Y$ is also generated by $[\overline{C_{R_i}}]$. There can not be any relations between these generators of the Chow ring, otherwise any relation would induce a relation in the homology group via the cycle class map.
		\begin{Corollary}
			$Y$ is a smooth compactification of $\mathds{A}^6_{\mathds{C}}$.
		\end{Corollary}
		There is an explicit algorithm to calculate the cells $C_{R_i}$ which is based on the following observations. 
		
		$\bullet$ If $R\in C_{R_{i}}$ and $z^2\in H_{R_{i}}$, then the coefficient matrix for $z$ in $R$ can be assume to be in the form $\begin{pmatrix}
			1 &0&0\\
			0&1&0
		\end{pmatrix}$ (since we only care about $R$ up to the action of $GL(2)\times GL(3)$) and $R$ can thus be assumed to be in the following form $R=\begin{pmatrix}
			z+A &B&C\\
			D&z+E&F
		\end{pmatrix}$, where $A,B,C,D,E,F$ are linear combinations of $x,y$. In the following, we use $(A)$ to denote the largest term of $A$ up to scalar (with respect to the order $z>y>x$). If $C$ and $F$ are linearly independent, then we can assume $(C)\neq(F)$ by elementary matrix transformations (i.e. after replacing $R$ by another $R'$ in the same orbit of $GL(2)\times GL(3)$ action), and then $R_{i}$ must be equal to $\langle z^2, z(C), z(F)\rangle$. If instead $C$ and $F$ are linearly dependent, then we just assume that $C=0$, and $R_{i}$ must be equal to $\langle z^2, z(F), (B)(F)\rangle$. 
		
		$\bullet$ If $z^2 \notin H_{R_{i}}$ but some degree one term in $z$ appears in $H_{R_{i}}$, then the coefficient matrix for $z$ in $R$ can be assumed to be $\begin{pmatrix}
			1 &0&0\\
			0&0&0
		\end{pmatrix}$, and $R$ can thus be assumed to be in the form $R=\begin{pmatrix}
			z+A &B&C\\
			D&E&F
		\end{pmatrix}$, where $A,B,C,D,E,F$ are linear combinations of $x,y$. If $E$ and $F$ are linearly independent, then we can assume $(E)\neq(F)$, and $H_{R_{\infty}}=\langle z(E),z(F),\mathrm{max}\{(F)(B),(E)(C)\}\rangle$ (in fact, by elementary matrix transformations, we can ensure also that $(E)\neq(B)$, so that $(F)(B)\neq(E)(C)$). Otherwise if $E$ and $F$ are linearly dependent, then we can assume that $F=0$, and by linear transformation we can assume that $(D)\neq(E)$ and $H_{R_{i}}$ must be $\langle z(E), (C)(D), (C)(E)\rangle$.
		
		$\bullet$ If there is not any term in $z$ appearing in $H_{R_{i}}$, then $R=R_i=\langle x^2,xy,y^2\rangle$.
		
		Now if we are given $R_{\infty}\coloneq\lim_{t\rightarrow\infty}t.R$, we can check among the above cases to see which form $R$ would take, and such kind of $R$ would be determined by the coefficients in $A,B,C,D,E,F$ in the last paragraphs, i.e. be given by a point in $\mathds{A}^N$ with $N$ possibly very large. Then we begin to consider the elementary matrix transforms which preserve the form of $R$ discussed in the last paragraphs. Notice that if $R$ and $R'$ are linked by such elementary matrix transformations, then $R$ and $R'$ represents the same point in $Y$ . This allows us to reduce $N$ to appropriate $n$ until we get a bijective parametrization of the Białynicki-Birula cell using $\mathds{A}^n$ (One can check whether the parametrization is injective by directly calculating $H_R$). 
		%For example , using these observations, it is possible to calculate the biggest cell:
		%\begin{Lemma}
		%The largest cell $C_{z^2\wedge zy\wedge zx}\cong \mathds{A}^6$ is in the form $$(a_1,a_2,a_3,a_4,a_5,a_6)\in \mathds{A}^6\mapsto \begin{pmatrix} z+a_1x & a_2x & y\\a_3y+a_4x &z+a_5y+a_6x & x\end{pmatrix}\in Y$$ It contains some elements in each of the orbits of $GL(W)$ action except the orbit of elements in the form $R=\langle x^2,xy,y^2\rangle$. \label{bigcell}
		%\end{Lemma}

		The next step is to explore the inclusion relations between the closures of cells. While doing this, it is useful to observe that any element of the form $z\rightarrow z+ay+bx, y\rightarrow y+cx, x\rightarrow x$ in $GL(W)$, where $a,b,c\in \mathds{C}$, preserves every cell. Explicit calculations tell us that, luckily in our case the cellular decomposition is in fact a stratification, which means that the closure of any cell is the disjoint union of open cells $\mathds{A}^{n_i}$ in the cellular decomposition. The incidence relations are exhibited in the following diagram (where, for example, $z^2\wedge zx\wedge zy$ actually means $\overline{C_{z^2\wedge zx\wedge zy}}$):
		
		
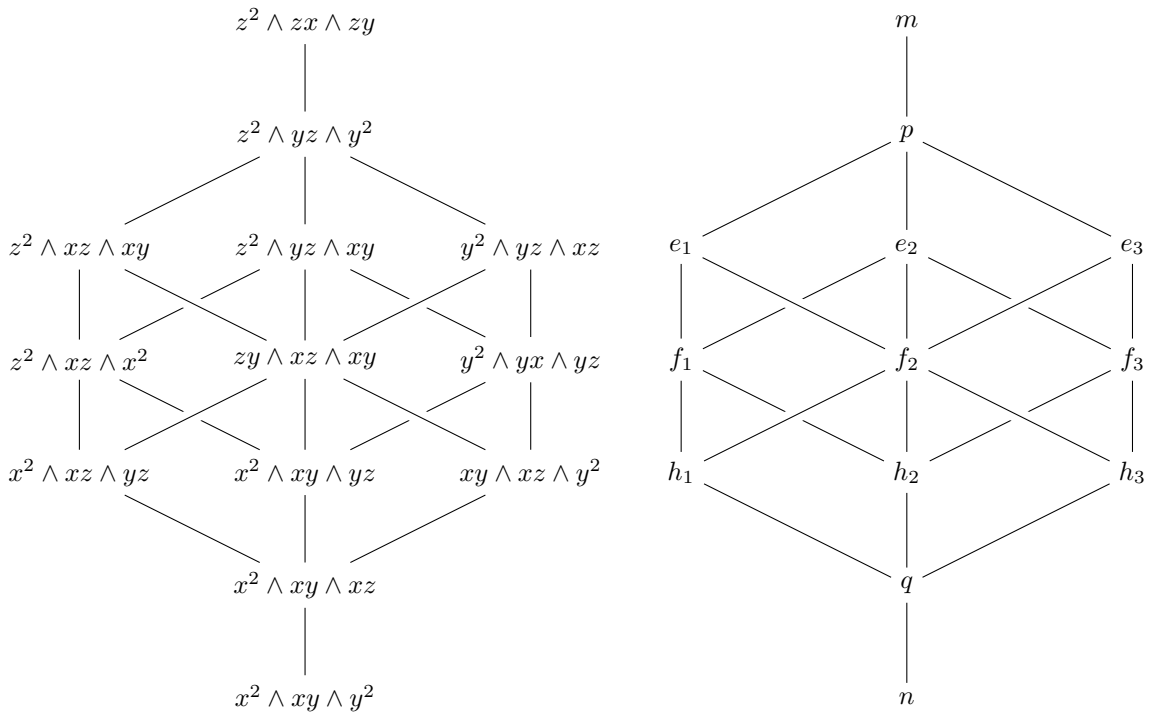
\begin{tikzpicture}
			\node (max) at (0,10.5) {$z^2\wedge zx\wedge zy$};
			\node (a) at (0,9) {$z^2\wedge yz\wedge y^2$};
			\node (b) at (-3,7.5) {$z^2\wedge xz\wedge xy$};
			\node (c) at (0,7.5) {$z^2\wedge yz\wedge xy$};
			\node (d) at (3,7.5) {$y^2\wedge yz\wedge xz$};
			\node (e) at (-3,6) {$z^2\wedge xz\wedge x^2$};
			\node (f) at (0,6) {$zy\wedge xz\wedge xy$};
			\node (g) at (3,6) {$y^2\wedge yx\wedge yz$};
			\node (h) at (-3,4.5) {$x^2\wedge xz\wedge yz$};
			\node (i) at (0,4.5) {$x^2\wedge xy\wedge yz$};
			\node (j) at (3,4.5) {$xy\wedge xz\wedge y^2$};
			\node (k) at (0,3) {$x^2\wedge xy\wedge xz$};
			\node (min) at (0,1.5) {$x^2\wedge xy\wedge y^2$};
			\draw (min) -- (k) -- (i) -- (f) -- (c) -- (a) -- (max)
			(a) -- (b) -- (e) -- (h) -- (k)
			(a) -- (d) -- (g) -- (j) -- (k);
			\draw (c) -- (e) -- (i) -- (g) -- (c);
			\draw[preaction={draw=white, -,line width=6pt}] (b) -- (f) -- (j);
			\draw[preaction={draw=white, -,line width=6pt}] (d) -- (f) -- (h);
			
			\node (m) at (8,10.5) {$m$};
			\node (p) at (8,9) {$p$};
			\node (e_1) at (5,7.5) {$e_1$};
			\node (e_2) at (8,7.5) {$e_2$};
			\node (e_3) at (11,7.5) {$e_3$};
			\node (f_1) at (5,6) {$f_1$};
			\node (f_2) at (8,6) {$f_2$};
			\node (f_3) at (11,6) {$f_3$};
			\node (h_1) at (5,4.5) {$h_1$};
			\node (h_2) at (8,4.5) {$h_2$};
			\node (h_3) at (11,4.5) {$h_3$};
			\node (q) at (8,3) {$q$};
			\node (n) at (8,1.5) {$n$};
			\draw (n) -- (q) -- (h_2) -- (f_2) -- (e_2) -- (p) -- (m)
			(p) -- (e_1) -- (f_1) -- (h_1) -- (q)
			(p) -- (e_3) -- (f_3) -- (h_3) -- (q);
			\draw (e_2) -- (f_1) -- (h_2) -- (f_3) -- (e_2);
			\draw[preaction={draw=white, -,line width=6pt}] (e_1) -- (f_2) -- (h_3);
			\draw[preaction={draw=white, -,line width=6pt}] (e_3) -- (f_2) -- (h_1);
		\end{tikzpicture}
		
		The first row is the only 6-dimensional cell, the second row is the only 5-dimensional cell, etc. (in the decreasing dimension order). The notations $m,n,p,q,e_i,f_i,h_i$ for the cells are indicated in the second parallel Hasse diagram above, for example $e_3$ means $\overline{C_{y^2\wedge yz\wedge xz}}$. Sometimes we will also use these notations to denote the fundamental classes of closures of cells.

		Such a Hasse diagram allows us to know explicitly what the closure of a cell is and thus allows us to compute the intersection numbers of the closures of cells. For example, to calculate the self-intersection number of (the closure of) a 3-dimensional cell, we need to move the cell to another position and intersect it with the cell in the original position, just as in the intersection theory of Schubert cycles in a Grassmannian. To do this, we pick an element $g$ in $GL(W)$ and change the action of the torus $\mathds{G}_m$ correspondingly: let $g: x\rightarrow x', \ y\rightarrow y', z\rightarrow z'$ be an element of $GL(W)$, we define a new action of $\mathds{G}_m$ on $Y$ correspondingly to be $t\rightarrow [x'\rightarrow x', \ y'\rightarrow ty', \ z'\rightarrow t^3z']$. The new action of the torus gives a new cellular decomposition. For example we have a new cell $C_{z'y'\wedge x'y'\wedge x'z'}$ and obviously the action of the element $g\in GL(W)$ on $Y$ will map the old cell $C_{zy\wedge xy\wedge xz}$ to this new cell. Since $g\in GL(W)$ can be connected with $\mathrm{Id}_W\in GL(W)$ in $GL(W)$ using finitely many rational curves, we know that $\overline{C_{z'y'\wedge x'y'\wedge x'z'}}$ and $\overline{C_{zy\wedge xy\wedge xz}}$ are not only algebraically equivalent, but actually rationally equivalent.
		
		As in the above paragraph, we pick $x'=z, \ y'=y, \ z'=x$ in the definition of $g\in GL(W)$ and we define the new action of $\mathds{G}_m$ on $Y$ correspondingly. It turns out that the cells of the new cellular decomposition will intersect the complementary dimensional cells of the old cellular decomposition transversally, and we can count how many intersection points there are. In the next section we will see that all the point classes are equal to 1 in $A^6(Y)=\mathds{Z}$. Now we have all the intersection numbers: 
		\begin{Proposition}
			The intersection number of $m$ with $n$ is equal to 1, the intersection number of $p$ with $q$ is equal to 1, $\{e_i\}$ and $\{h_i\}$ are dual bases with respect to the intersection pairing, and $\{f_i\}$ is an orthonormal basis with respect to the intersection pairing. \label{cell}
		\end{Proposition}
		
		As an example of the algorithm for calculating Biaynicki-Birula cells, we have:
		\begin{Lemma}
			The biggest cell $C_{z^2\wedge zy\wedge zx}\cong \mathds{A}^6$ is in the form $$(a_1,a_2,a_3,a_4,a_5,a_6)\in \mathds{A}^6\mapsto \begin{pmatrix} z+a_1x & a_2x & y\\a_3y+a_4x &z+a_5y+a_6x & x\end{pmatrix}\in Y$$ It contains some elements in each of the orbits of the $GL(W)$-action except the orbit $\mathcal{O}_2'$. \label{bigcell}\end{Lemma}
		The Corollary of the above Lemma is that for each element outside the orbit $\mathcal{O}_2'$, we have an explicit $\mathds{A}^6$ neighbourhood of this element by taking the $SL(W)$-translate of the above $\mathds{A}^6$. And if you want, by choosing another suitable torus action and calculating the Białynicki-Birula cells, you can also derive explicit $\mathds{A}^6$ neighbourhoods for elements in the orbit $\mathcal{O}_2'$.
		\subsection{Singularity of the Special Divisor}
		We let $D$ denote the complement of the biggest cell in $Y$ and call it to be the special divisor. The divisor $D$ is a hyperplane section of $Y$ because $D$ consists of those $R$ such that the coefficient of $z^2\wedge zy\wedge zx$ in $[\wedge^3H_R]\in \mathds{P}(\wedge^3S^2V)$ is zero. In the problem list of \cite{Hirzebruch}, the 27-th problem is about classifying all smooth compactifications $X$ of $\mathds{C}^n$ with $b_2(X)=1$. It was proven in \cite{peternell2024compactificationscncomplexprojective} that if the complement $D$ of $\mathds{C}^n$ in $X$ is a smooth hypersurface, then $X\cong\mathds{P}^n_{\mathds{C}}$. In this subsection, we investigate the singularity of the irreducible hypersurface $D$.
		\begin{Lemma}
			If $z^2\notin H_{R_{\infty}}$, then every $R\in C_{R_{\infty}}$ is a singular point of $D$.
			\label{thm 4.23}
		\end{Lemma}
		\begin{proof}
			We first assume that $R\in C_{R_{\infty}}$ is not in the form $\langle x'^2,x'y',y'^2\rangle$, then it has an affine neighbourhood as in \cref{bigcell}.
			In view of this, we can write down the defining equation of the hyperplane $D  $ as a polynomial in coordinates of $\mathds{A}^{6}$, and check whether the expansion of this polynomial at the point $R$ has nonzero linear terms. Since we only care about whether there are nonzero linear terms, it doesn't matter if we verify this condition after adding more affine coordinates and performing some linear transformations of $\mathds{A}^6$. Assuming $z^2\notin H_{R_{\infty}}$, we can take a representative of $R$ to be $\begin{pmatrix}
				z+A &B&C\\
				D&E&F
			\end{pmatrix}$ (assuming also there is some degree one term in z in $H_{R_{\infty}}$), then we only need to verify that for $\epsilon\in \mathbb{C}[\epsilon]/(\epsilon^2)$ and $R_{\epsilon}=\begin{pmatrix}
				z+A &B&C\\
				D&E&F
			\end{pmatrix}+\epsilon\begin{pmatrix}
				A_{1,1}&A_{1,2}&A_{1,3}\\
				A_{2,1}&A_{2,2}&A_{2,3}
			\end{pmatrix}$
			where $A_{i,j}=a_{i,j,1}x+a_{i,j,2}y+a_{i,j,3}z$ and $a_{i,j,k}$ are affine coordinates for $\mathds{A}^{18}$, the coefficient of the term $z^2\wedge zy\wedge zx$ in $H_{R_{\epsilon}}$ is zero. The latter claim is easy to check from the expression of $R_{\epsilon}$ (remember that $\epsilon^2=0$ and $A,B,C,D,E,F$ are linear combinations of $x,y$).
			
			Now we deal with the case where $R$ is in the form $\langle x'^2,x'y',y'^2\rangle$ and $z^2\notin H_{R_{\infty}}$. It is obvious that $R$ can only be $\langle x^2,xy,y^2\rangle$. Notice that $R$ is the limit of $\langle x^2,xy,y^2-txz\rangle$ when $t\rightarrow 0$. The above paragraph shows that $\langle x^2,xy,y^2-txz\rangle_{t\neq0}$ lies in the singular locus of $D$, and we can thus conclude that $\langle x^2,xy,y^2\rangle$ also lies in the singular locus of $D$.
		\end{proof}
		
		We consider the curve contained in $Y$ passing through $\langle x^2, xz,z^2\rangle$: $t\rightarrow \begin{pmatrix}
			z & x & ty\\
			0 & z & x
		\end{pmatrix}$. This curve intersects transversally with the defining hyperplane for $D$ at the point $\langle x^2, xz,z^2\rangle$. We thus conclude that $\langle x^2, xz,z^2\rangle$ is a smooth point of $D$. Since the singular locus is closed and invariant under the torus action that we pick, we conclude that $C_{x^2\wedge xz\wedge z^2}$ is contained in the smooth locus of $D$. Similar arguments show that $C_{z^2\wedge xz\wedge xy}$, $C_{z^2\wedge yz\wedge xy}$ and $C_{z^2\wedge yz\wedge y^2}$ are all contained in the smooth locus of $D$. Based on the Hasse diagram, we have the following Proposition:
		\begin{Proposition}
			The singular locus of $D$ is irreducible of dimension 4 and is equal to $\overline{C_{y^2\wedge yz\wedge xz}}=\overline{e_3}$.
		\end{Proposition}

		\section{Complete Description of Chow ring}
		Now we study the Chow ring of $Y$, where $Y$ is viewed as the fine quiver moduli space of representations of the 3-Kronecker quiver with dimension vector $(2,3)$. We can apply known results on tautological presentations of the Chow rings of fine quiver moduli spaces. By the result of \cite{Kingwalter}, we know that the Chow ring $A^*(Y)$ is generated by the Chern classes of $\mathcal{U}_2^*, \mathcal{U}_1^*$ as a $\mathds{Z}$-algebra. And by \cite{Franzen}, if we work over $\mathds{Q}$, all the relations between these generators are just $c_1(\mathcal{U}_1^*)=c_1(\mathcal{U}_2^*)$ together with those relations coming from the so-called forbidden dimension vectors.
		
		We denote $c_i(\mathcal{U}_2^*)$ by $c_i$ and we denote $c_i(\mathcal{U}_1^*)$ by $d_i$. By calculation (see the last example in \cite{Franzen}), we get the following description: $$A^1(Y)=\mathds{Z}c_1=\mathds{Z}d_1,$$
		$$A^2(Y)=\mathds{Z}c_1^2+\mathds{Z}c_2+\mathds{Z}d_2,$$ $$A^3(Y)=\mathds{Z}c_1c_2+\mathds{Z}c_1d_2+\mathds{Z}c_3,$$ with $c_1^3=4c_1d_2-3c_3$,
		$$A^4(Y)=\mathds{Z}c_2^2+\mathds{Z}c_2d_2+\mathds{Z}d_2^2,$$ with $c_1^4=-3c_2^2+9c_2d_2+3d_2^2, \ c_1^2c_2=3d_2^2+c_2d_2, \ c_1^2d_2=3d_2^2, \ c_1c_3=c_2^2-3c_2d_2+3d_2^2$,
		$$A^5(Y)=\mathds{Z}(3c_3d_2-c_1^2c_3)=\mathds{Z}\frac{c_2c_3}{3},$$
		with $c_3d_2=\frac{2}{3}c_2c_3, \ c_1^2c_3=\frac{5}{3}c_2c_3, \ c_1^5=19c_2c_3, \ c_1^3c_2=9c_2c_3, \ c_1^3d_2=6c_2c_3, \ c_1c_2^2=\frac{14}{3}c_2c_3, \ c_1d_2^2=2c_2c_3, \ c_1c_2d_2=3c_2c_3$,
		$$A^6(Y)=\mathds{Z}c_3^2,$$ with $c_1^6=57c_3^2, \ c_1^4c_2=27c_3^2, \  c_1^4d_2=18c_3^2, \ c_1^3c_3=5c_3^2, \ c_1^2c_2^2=14c_3^2, \ c_1^2d_2^2=6c_3^2,  \ c_1^2c_2d_2=9c_3^2, \ c_1c_3d_2=2c_3^2, \ c_1c_2c_3=3c_3^2, \ c_2^3=9c_3^2, \ c_2^2d_2=5c_3^2, \ c_2d_2^2=3c_3^2, \ d_2^3=2c_3^2$.
		
		By using the formula given in \cite{belmans2023chow}, we know that the point class is $c_3^2$. Thus we have the following list of intersection numbers:
		\begin{Proposition}
			$c_1^6=57, \ c_1^4c_2=27, \  c_1^4d_2=18, \ c_1^3c_3=5, \ c_1^2c_2^2=14, \ c_1^2d_2^2=6,  \ c_1^2c_2d_2=9, \ c_1c_3d_2=2, \ c_1c_2c_3=3, \ c_2^3=9, \ c_2^2d_2=5, \ c_2d_2^2=3, \ d_2^3=2$.\label{intersection}
		\end{Proposition}
		In particular, the degree of $Y$ is 57.
		
		Now we can give a presentation of the Chow ring of $Y$:
		\begin{thm}\label{chow presented}
			$A^*(Y)_{\mathds{Q}}$ is generated by $c_1,c_2,d_2$, and the relations between them are \begin{equation*}
				c_1^4=-3c_2^2+9c_2d_2+3d_2^2, \ c_1^2c_2=3d_2^2+c_2d_2, \ c_1^2d_2=3d_2^2
			\end{equation*}
			\begin{equation*}
				9c_1c_2^2=14c_1c_2d_2, \ 3c_1d_2^2=2c_1c_2d_2
			\end{equation*}
		\end{thm}
		\begin{proof}
			Noticing that $c_1^3=4c_1d_2-3c_3$, we conclude that the Chow ring is generated by $c_1,c_2,d_2$ over $\mathds{Q}$. It is a routine check that these relations are all the relations needed. 
		\end{proof}

		Now we calculate the fundamental classes of the closures of cells coming from the Białynicki-Birula decomposition. To do this, we only need to calculate the intersection numbers of the closures of cells with products of Chern classes, because the intersection pairing is nondegenerate.
		
		First we can calculate the degrees of closures of cells using the Software Macaulay 2: we have the algorithm above to explicitly calculate the open cells, and the degrees of the closures of open cells are equal to the number of solutions of certain polynomial equations corresponding to the equations defining some general hyperplane sections. The result of the calculation is the following: (The meanings of $e_i,f_i,h_i$ are the same as in the last section.
		\begin{Proposition}
			The degrees of $h_1,h_2,h_3$ are equal to $2,1,2$ respectively. The degrees of $f_1,f_2,f_3$ are equal to $4,5,4$ respectively. The degrees of $e_1,e_2,e_3$ are equal to $9,21,9$ respectively. \label{degree}.
		\end{Proposition}
		
		With the algorithm mentioned in the last section, it is easy to see that $h_2$ is actually isomorphic to a 2-dimensional linear subspace of $\mathds{P}(\wedge^3(S^2W))$, and the pullback of $\mathcal{U}_2^*$ to $h_2$ has a rank 2 trivial direct summand, and the pullback of $\mathcal{U}_1^*$ to $h_2$ has a rank 1 trivial direct summand. As a result, $c_2.h_2=d_2.h_2=0$.
		
		To calculate the intersections of $c_2$ with $h_1,h_3$, we pick 2 general elements in $S^2W^*$ and view these two elements as two (maybe not general) global sections of $\mathcal{U}_2^*|_{h_i}$ via the map $S^2W^*\rightarrow\mathcal{U}_2^*$ and count over how many points these 2 global sections are linearly dependent. To justify that the number of such points is exactly the intersection number, we recall that we have the embedding $Y\hookrightarrow \mathrm{Gr}(3,S^2W)$ and $\mathcal{U}_2^*$ is the pullback of the dual of the universal subbundle on $\mathrm{Gr}(3,S^2W)$. Now 2 general elements in $S^2W^*$ are 2 general global sections of the dual of the universal subbundle on $\mathrm{Gr}(3,S^2W)$, and the locus where these 2 global sections are linearly dependent is a Schubert cycle which will intersect $h_i$ transversally at any intersection point by Kleiman's Theorem (and will not intersect the boundary of $h_i$).
		
		We calculate for example the intersection number $c_2.h_3$. Recall that $h_3$ is the class of the closure of the cell associated to the torus action fixed point $xy\wedge xz\wedge y^2$. The open cell is isomorphic to $\mathds{A}^2$, and the element $R$ with coordinate $(a,b)\in\mathds{A}^2$ is identified with $R=\begin{pmatrix}
			z+ax & 0 & y+bx\\
			y & x & 0
		\end{pmatrix}$ and $H_R=\langle zx+ax^2,xy+bx^2,y^2+bxy\rangle\subset S^2W$ using the algorithm in the last section. Let us take two general elements $p,q\in S^2W^*$, then $p,q$ are linearly dependent over the point $R$ iff
		$\mathrm{Ker}(p)\cap H_R=\mathrm{Ker}(q)\cap H_R$. Let us assume that $p(zx)=p_1, \ p(x^2)=p_2, \ p(xy)=p_3, \ p(y^2)=p_4$ and similarly $q(zx)=q_1, \ q(x^2)=q_2, \ q(xy)=q_3, \ q(y^2)=q_4$. Then $\mathrm{Ker}(p)\cap H_R=\mathrm{Ker}(q)\cap H_R$ iff $\mathrm{rank}\begin{pmatrix}
			p_1+ap_2 & bp_2+p_3 & bp_3+p_4\\
			q_1+aq_2 & bq_2+q_3 & bq_3+q_4
		\end{pmatrix}\leq 1$. Obviously when $p_i,q_i$ are general, there are exactly 2 elements in the cell $\mathds{A}^2$ such that the rank of the above matrix is $\leq 1 $. We conclude that $c_2.h_3=2$.
		
		Similar calculation gives $c_2.h_1=1$. 
		And similarly we can calculate $d_2.h_i$ by picking a general element of $S^{2,1}W^*$, viewing it as a global section of $\mathcal{U}_1^*|_{h_i}$ and counting the number of points in $h_i$ over which the global section vanishes. We find $d_2.h_1=d_2.h_3=1$.
		
		With the above information and the intersection numbers in \cref{intersection}, we get the following Proposition:
		\begin{Proposition}
			$h_1=c_2^2-4c_2d_2+4d_2^2, \ h_2=-c_2^2+3c_2d_2-2d_2^2, \ h_3=c_2^2-2c_2d_2+d_2^2, \ e_1=-c_2+2d_2, \ e_2=c_1^2-2d_2, \ e_3=c_2-d_2$.
		\end{Proposition}
		\begin{proof}
			$\{e_i\}$ are calculated based on the fact that it is the dual basis of $\{h_i\}$ with respect to the intersecting pairing by \cref{cell}.
		\end{proof}
		
		Similar computations give $c_3.f_1=c_3.f_3=0, \ c_3.f_2=1$. It follows that $c_3=f_2$. 
		
		From \cref{degree} and $c_1^3=4c_1d_2-3c_3$, we deduce that $c_1d_2.f_1=c_1d_2.f_3=1, \ c_1d_2.f_2=2$, thus $c_1d_2=f_1+2f_2+f_3$.
		
		Now the remaining expression to derive is the expression of $c_1c_2$ in terms of $f_1,f_2,f_3$. Assume that $c_1c_2=\lambda_1f_1+\lambda_2f_2+\lambda_3f_3$, the fact that $c_1c_2.c_3=3$ (and $c_3=f_2$) tells us that $\lambda_2=3$. We also know that $\ c_1c_2.c_1d_2=9, \ c_1c_2. c_1c_2=14, \ \lambda_i=c_1c_2. f_i\geq 0 $, then we have $\lambda_1+2\lambda_2+\lambda_3=9$ and $\lambda^2_1+\lambda^2_2+\lambda^2_3=14$. Consequently, we deduce that either $c_1c_2=f_1+3f_2+2f_3$ or $c_1c_2=2f_1+3f_2+f_3$.
		A calculation using the explicit parametrization of the interior of $f_1$ shows that $c_1c_2.f_1=1$, so the former case actually holds.
		
		\begin{Proposition}
			$f_1=-c_1c_2+2c_1d_2-c_3, \ f_2=c_3, \ f_3=c_1c_2-c_1d_2-c_3$.
		\end{Proposition}
		
		Now we determine the fundamental classes of the orbit closures.
		
		We can see that the closure of the orbit $\mathcal{O}_5$ intersects the line $(s:t)\in \mathds{P}^1\rightarrow yx\wedge yz\wedge (sx^2+tz^2)$ transversally at two points, which shows that the fundamental class of the 5 dimensional orbit closure is equal to $2c_1$.

		The orbit $\mathcal{O}_4$ is 4-dimensional and it contains a dense subset parametrized as follows:
		$$(a,b,c,d)\in \mathds{A}^3\times \mathds{A}^*\mapsto \langle(x+ay+bz)^2, (x+ay+bz)(y+cz), (y+cz)^2-(x+ay+bz)(dz)\rangle$$ Using this dense parametrization of the orbit closure, we can check that the intersection numbers of the orbit closure with $d_2^2,d_2c_2,c_2^2$ are equal to 6,9,15 respectively. This shows that the fundamental class of this orbit closure is equal to $3d_2$.

		By calculations using the explicit description of the universal representation map $\mathcal{U}_1\rightarrow\mathcal{U}_2\otimes W$ given in \ref{univrep}, the degenerate locus of the morphism between vector bundles $\mathcal{U}_2^*\otimes \mathcal{U}_1^*\rightarrow (\mathcal{U}_1^*\otimes W)\otimes \mathcal{U}_1^*\rightarrow S^2\mathcal{U}_1^*\otimes W$ is set-theoretically the orbit $\mathcal{O}_2$. Since the degenerate locus has the correct codimension, we know that the class of the degenerate locus is $-3c_2d_2+6d_2^2$. One can calculate that $(c_1^2. -3c_2d_2+6d_2^2)=9$. On the other hand, the degree of the orbit $\mathcal{O}_2$ is also equal to 9 (the same calculation as in \cref{degree}). We conclude that this orbit is the scheme-theoretical degenerate locus of $\mathcal{U}_2^*\otimes \mathcal{U}_1^*\rightarrow (\mathcal{U}_1^*\otimes W)\otimes \mathcal{U}_1^*\rightarrow S^2\mathcal{U}_1^*\otimes W$ and its fundamental class is $-3c_2d_2+6d_2^2$. 
		
		Using the same argument, we know that the orbit $\mathcal{O}_2'$ is the scheme-theoretical degenerate locus of $\mathcal{U}_2^*\otimes \mathcal{U}_1\rightarrow W$ and its fundamental class is $3c_2d_2-3d_2^2$.
		
		\subsection{Symmetry and Outer Automorphism}
		
		We observe the following symmetry in the `$c_1$ multiplication' (in the Chow ring) with Białynicki-Birula cells:
		
		
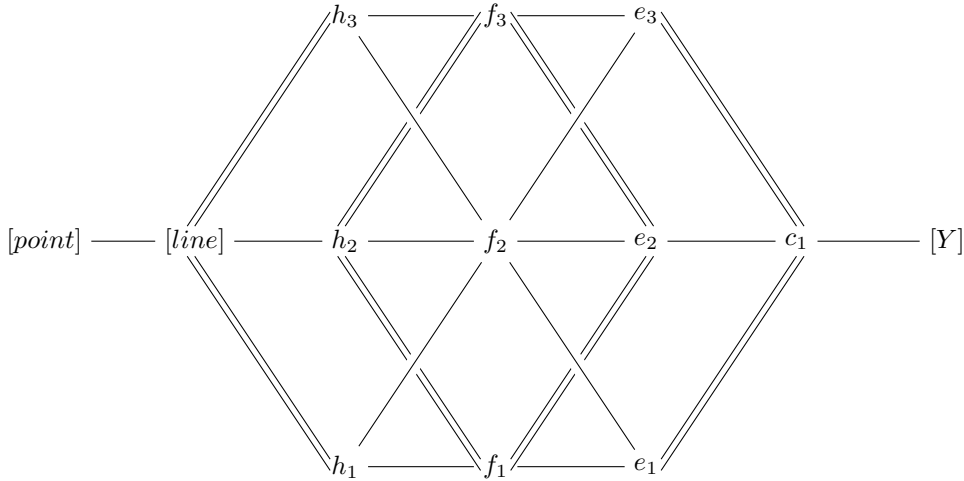
\begin{tikzpicture}
			\node (m) at (14,3) {$[Y]$};
			\node (p) at (12,3) {$c_1$};
			\node (e_1) at (10,0) {$e_1$};
			\node (e_2) at (10,3) {$e_2$};
			\node (e_3) at (10,6) {$e_3$};
			\node (f_1) at (8,0) {$f_1$};
			\node (f_2) at (8,3) {$f_2$};
			\node (f_3) at (8,6) {$f_3$};
			\node (h_1) at (6,0) {$h_1$};
			\node (h_2) at (6,3) {$h_2$};
			\node (h_3) at (6,6) {$h_3$};
			\node (q) at (4,3) {$[line]$};
			\node (n) at (2,3) {$[point]$};
			\draw (n) -- (q) -- (h_2) -- (f_2) -- (e_2) -- (p) -- (m)
			(e_1) -- (f_1) -- (h_1)
			(e_3) -- (f_3) -- (h_3);
			\draw (3.9,3.2)--(5.8,6.05);
			\draw (4,3.2)--(5.8,5.9);
			\draw (4,2.8)--(5.8,0.1);
			\draw (3.9,2.8)--(5.8,-0.05);
			\draw (5.9,3.2)--(7.8,6.05);
			\draw (6,3.2)--(7.8,5.9);
			\draw (6,2.8)--(7.8,0.1);
			\draw (5.9,2.8)--(7.8,-0.05);
			\draw (8.2,6.05)--(10.1,3.2);
			\draw (8.2,5.9)--(10,3.2);
			\draw (8.2,-0.05)--(10.1,2.8);
			\draw (8.2,0.1)--(10,2.8);
			\draw (10.2,6.05)--(12.1,3.2);
			\draw (10.2,5.9)--(12,3.2);
			\draw (10.2,-0.05)--(12.1,2.8);
			\draw (10.2,0.1)--(12,2.8); 
			\draw[preaction={draw=white, -,line width=6pt}] (e_1) -- (f_2) -- (h_3);
			\draw[preaction={draw=white, -,line width=6pt}] (e_3) -- (f_2) -- (h_1);
		\end{tikzpicture}
		
		In the diagram, the number of edges between two classes is the coefficient of the larger degree class in the product of $c_1$ with the smaller degree class. For example, we read from it that $c_1.f_1=2h_2+h_1$. 
		
		Now we observe an obvious upside-down symmetry in this diagram, where $h_3,f_3,e_3$ seems to be symmetric to $h_1,f_1,e_1$ respectively. There is an involution of $Y$ which is responsible for this symmetry and we will discuss this involution in this subsection. 
		
		We consider the following nondegenerate bilinear form on $W$ which is equal to $\begin{pmatrix}
			0&0&1\\
			0&1&0\\
			1&0&0
		\end{pmatrix}$ when written as a matrix with respect to the basis $\{x,y,z\}$. This gives us another isomorphism $W\cong W^*$, and thus another isomorphism $Y^*\cong Y$. If we compose this isomorphism with $f:Y\rightarrow Y^*$ constructed in \cref{automorphismsection}, we get another automorphism $g_2$ of $Y$. It is obvious that $g_2$ is equal to the composition of $g_1$ with the automorphism induced by the element $x\rightarrow z,y\rightarrow y, z\rightarrow x$ in $GL(W)$.

		The interesting part of $g_2$ is that it gives us the upside-down symmetry in the diagram presented in the beginning of this subsection. One can show that $g_2$ permutes the limit points of $\{e_1,e_3\},\{f_1,f_3\},\{h_1,h_3\},\{m,p\}$, $\{q,n\}$ pairwise and fixes the other limit points of Białynicki-Birula cells. We consider another torus action $x\rightarrow x, y\rightarrow t^2y,z\rightarrow t^3z$ on $Y$. This new torus action has the same 13 fixed points. One can observe that 
		
		$\bullet$
		If $\lim_{t\rightarrow\infty} t.R=R_{\infty}$ for the old torus action, then $\lim_{t\rightarrow\infty} t.g_2(R)=g_2(R_{\infty})$ for the new torus action.
		
		$\bullet$ If $\lim_{t\rightarrow\infty} t.R$ is equal to the limit point of one of the cell closures $e_i,f_i,h_i$ for the new torus action, then the value of $\lim_{t\rightarrow\infty} t.R$ is independent of which torus action we pick among the two. 
		
		Now the symmetry goes as follows: for example, if $R$ lies in the interior of $e_1=\overline{C_{z^2\wedge xz\wedge xy}}$, then  $\lim_{t\rightarrow\infty}t.g_2(R)=g_2(\langle z^2,xz,xy\rangle)=\langle y^2,yz,xz\rangle$ for the new torus action by the first observation. As a result, $\lim_{t\rightarrow\infty}t.g_2(R)=\langle y^2,yz,xz\rangle$ for the old torus action by the second observation, which means $g_2(R)$ is in the interior of $e_3=\overline{C_{y^2\wedge yz\wedge xz}}$. This shows that $g_2$ exchanges the interiors of $e_1$ and $e_3$.
		
		To summarize,  we have the following Proposition:
		\begin{Proposition}\label{involution}
			$g_2$ exchanges the interiors of $e_1$ and $e_3$, the interiors of $f_1$ and $f_3$, the interiors of $h_1$ and $h_3$. And $g_2$ maps the interiors of $e_2,f_2,h_2$ to themselves.
		\end{Proposition}
		
		As a result, the action of $g_2$ on $A^*(Y)_{\mathds{Q}}$ can be described as follows: $d_2\coloneq c_2(\mathcal{U}_1^*)(=e_1+e_3)$ and $c_1$ are fixed by the involution $g_2$. Meanwhile, $c_2\coloneq c_2(\mathcal{U}_2^*)$ and $3d_2-c_2$ are interchanged by the involution $g_2$.
		
		\section{Quantum Cohomology Ring}
		We let $(QH^*(Y),*)$ denote the small quantum cohomology ring of $Y$ with $\mathds{C}$-coefficients. As a vector space, it is equal to $H^*(Y,\mathds{C})\otimes _{\mathds{C}}\mathds{C}[q]$, where $q$ is the formal quantum parameter. The quantum product $*$ is defined using Gromov-Witten invariants and makes the above vector space an associative ring. The ring is graded with respect to the quantum product $*$ by letting $q$ be of degree 3 (which is the index of $Y$). 
		
		We let $I_{0,3,d}(-,-,-)$ denote the Gromov-Witten invariants, where the subscripts mean that we are considering stable 3-pointed maps of genus 0 and degree $d$. Take $\{T_i\}$ to be a basis of $H^*(Y,\mathds{C})$ and let $\{T_i^*\}$ be the dual basis in $ H^*(Y,\mathds{C})$ with respect to the cup product $\cup$. For arbitrary elements $\alpha,\beta \in H^*(Y,\mathds{C})$, the (small) quantum product is defined as (the $\mathds{C}[q]$-linear extension of) the following:
		$$\alpha*\beta=\alpha\cup\beta+\Sigma_{n\geq 1}\Sigma_i I_{0,3,n}(\alpha,\beta,T_i)T_i^*q^{n}.$$
		
		As said before, the quantum product respects the degree and thus $n.\mathrm{deg}(q)=\mathrm{deg}(q^n)\leq\mathrm{deg}(\alpha)+\mathrm{deg}(\beta)$ which tells us that the above sum is actually a finite
		sum.
		%(this is essentially the degree axiom for Gromov-Witten invariants).
		\begin{Lemma}[\cite{Tian}]\label{Tian}
			Assume that the ordinary cohomology ring $H^*$ has the presentation $H^*\cong \mathds{C}[x_1,x_2,\cdots,$
			$x_m]/(R_1,$
			$R_2,\cdots,R_s)$ in terms of generators and relations, where $R_i$ are polynomials in $\{x_j\}$. Then $QH^*\cong\mathds{C}[x_1,x_2,$
			$\cdots,x_m,q]/(\Tilde{R}_1,$
			$\tilde{R}_2,\cdots,\tilde{R}_s)$, where $\tilde{R}_i$ is an explicit deformation of the relation $R_i$.
		\end{Lemma}
		\begin{proof}[Idea of Proof]
			For arbitrary homogeneous elements $\alpha, \beta \in H^*$ with nonzero degrees, we know that $\alpha\cup\beta=\alpha*\beta-q(\Sigma_{n\geq 1}\Sigma_i I_{0,3,n}(\alpha,\beta,T_i)T_i^*q^{n-1})$. To show that $\alpha\cup\beta$ can be expressed as a polynomial in $\{x_i\}\cup\{q\}$ with respect to the quantum product $*$, we only need to prove that $\alpha,\beta, \Sigma_{n\geq 1}\Sigma_i I_{0,3,n}(\alpha,\beta,T_i)T_i^*q^{n-1}$ can be expressed as polynomials in $\{x_i\}\cup\{q\}$ with respect to the quantum product $*$. However $\alpha,\beta,$
			$\Sigma_{n\geq 1}\Sigma_i I_{0,3,n}(\alpha,\beta,T_i)T_i^*q^{n-1}$ all have lower degree than $\alpha\cup\beta$, and we can use induction on the degree to show that $\{x_i\}\cup\{q\}$ generate $QH^*$. Moreover, the above argument actually provides an algorithm to exhibit polynomials in $\mathds{C}[x_1,x_2,\cdots,$
			$x_m]$ (with respect to cup products) as polynomials in $\{x_i\}\cup\{q\}$ with respect to the quantum product $*$.
			
			The relation $R_i$ can be viewed as an element in $\mathds{C}[x_1,x_2,\cdots,$
			$x_m]$ and thus we can use the algorithm in the last paragraph to express $R_i$ as a polynomial $\tilde{R}_i$ in $\{x_i\}\cup\{q\}$ with respect to the quantum product $*$. It is also by induction that we can prove $\{\tilde{R}_i\}$ are all the relations needed.
		\end{proof}
		When we specialize to the case $q=0$, the small quantum cohomology ring is exactly the ordinary cohomology ring. 
		Hence, the small quantum cohomology ring can be viewed as a deformation family of the ordinary cohomology ring, parametrized by $q\in \mathds{C}=H^2(Y,\mathds{C})$.
		Apart from the small quantum cohomology ring, a big quantum cohomology ring can also be constructed. The big quantum cohomology ring can be viewed as a deformation family of the ordinary cohomology ring, parametrized by $H^*(Y,\mathds{C})\times H^2(Y,\mathds{C})$.
		\begin{definition}
			We say that the small quantum cohomology ring is generically semisimple if $QH^*(Y)|_{q=a}$ is semisimple (i.e. reduced) for general $a\in H^2(Y,\mathds{C})=\mathds{C}$. Similarly, the big quantum cohomology ring is generically semisimple if its specialization to a general value in  $H^*(Y,\mathds{C})\times H^2(Y,\mathds{C})$ is semisimple.
		\end{definition}
		
		\begin{Lemma}
			$QH^*(Y,\mathds{C})$ is reduced if and only if it is generically semisimple.
		\end{Lemma}
		\begin{proof}
			We first assume that $QH^*(Y,\mathds{C})$ is generically semisimple. Let $f\in QH^*(Y,\mathds{C})$ be a nilpotent element in $QH^*(Y,\mathds{C})$. Then for a general value $a\in \mathds{C}$, $f$ should be zero in $QH^*(Y)|_{q=a}$. Assume that $f$ corresponds to $\Sigma_{i=1}^{m}h_iq^i $ via the vector space isomorphism $H^*(Y,\mathds{C})\otimes _{\mathds{C}}\mathds{C}[q]$, then $f=0\in QH^*(Y)|_{q=a}$ is equivalent to $\Sigma_{i=1}^{m}h_ia^i=0\in H^*(Y,\mathds{C}) $. Since the latter condition holds for infinitely many $a\in \mathds{C}$, we conclude that $h_i=0$ and thus $f=0$.
			
			Now we assume that $QH^*(Y)$ is reduced. We notice that the localization ring $QH^*(Y)_{S}$ with respect to the multiplicatively closed subset $S=\{r\in QH^*(Y)|0\neq r\in \mathds{C}[q]\subset QH^*(Y)\}$ is also reduced because there is not any zerodivisor in this multiplicatively closed subset. Then we apply \cite[\href{https://stacks.math.columbia.edu/tag/0578}{Tag 0578}]{stacks-project} to the morphism $\mathrm{Spec}(QH^*(Y))\rightarrow\mathrm{Spec}(\mathds{C}[q])$ to conclude that $QH^*(Y)$ is generically semisimple.
		\end{proof}
		\begin{remark}\label{bigsemisimple}
			If $QH^*(Y,\mathds{C})$ is reduced or generically semisimple, then the big quantum cohomology ring is generically semisimple. This essentially follows from the definition and \cite[\href{https://stacks.math.columbia.edu/tag/0C0E}{Tag 0C0E}]{stacks-project} applied in a way similar to the proof of the Lemma above.
		\end{remark}
		Let $\mathcal{M}_{0,3}(Y,1)$ denote the Kontsevich moduli space of stable 3-pointed maps of genus 0 and degree 1. Let $\mathcal{M}^*_{0,3}(Y,1)$ be the open subset consisting of stable maps with irreducible sources and with all the marked points mapped into the open orbit $\mathcal{O}_6$ of $Y$. 
		
		\begin{Proposition}
			$\mathcal{M}_{0,3}(Y,1)$ is smooth irreducible of dimension 9. $\mathcal{M}^*_{0,3}(Y,1)$ is an open dense subvariety of $\mathcal{M}_{0,3}(Y,1)$. In particular, the virtual fundamental class is equal to the ordinary fundamental class.  
		\end{Proposition}
		\begin{proof}
			In the construction of Kontsevich moduli space given in \cite{FP97}, we embed $Y$ in some projective space $\mathds{P}^N$ and construct $\mathcal{M}_{0,3}(Y,1)$ as a closed subscheme of $\mathcal{M}_{0,3}(\mathds{P}^N,1)$. We know that $\mathcal{M}_{0,3}(\mathds{P}^N,1)$ is actually a locally trivial fibration over $\mathds{G}(1,\mathds{P}^N)$, with fibers being $\mathcal{M}_{0,3}(\mathds{P}^1,1)$, which is smooth of dimension 3. Consequently, we have a Cartesian diagram:
			$$\xymatrix{
				\mathcal{M}_{0,3}(Y,1) \ar[d]^{} \ar[r]_{}
				&  \mathcal{M}_{0,3}(\mathds{P}^N,1)\ar[d]_{}    \\
				F_1(Y) \ar[r]^{}     &    \mathds{G}(1,\mathds{P}^N)          }$$
			The first and the second statement of this Proposition then follow from the fact that $F_1(Y)$ is smooth of dimension 6 and a general line $l\subset Y$ intersects the open orbit in $Y$.
			The last statement is a direct consequence of the fact that $\mathcal{M}_{0,3}(Y,1)$ has the expected dimension together with \cite[Proposition 2]{YP-Lee}.
		\end{proof}
		
		To calculate the Gromov-Witten invariants, we need to establish the enumerativity of these invariants. For homogeneous spaces, this is done using Kleiman's Theorem. For our prehomogeneous space $Y$, we can make use of the following Lemma: 
		\begin{Lemma}[\cite{Graber}, Lemma 2.5]     \label{graber}
			Let $X$ be a variety endowed with an action of a connected algebraic group $G$ with only finitely many orbits and $Z$ be an irreducible scheme with a morphism $f : Z \rightarrow X$. Let $W$ be a subvariety of $X$ that intersects the orbit stratification properly (i.e., for any $G$-orbit $O$ of $X$, we have codim$_O(W \cap O)$ = codim$_X (W)$. Then there exists a dense open subset $U$ of $G$ such that for every $g\in U$, $f^{-1}(gW)$ is either empty or has pure dimension dim $W$$+$dim $Z$$-$dim $X$. Moreover, if $X, W $and $Z$ are smooth and we denote it by $W_{\mathrm{reg}}$ the subset of $W$ along which the intersection with the stratification is transverse, then the (possibly empty) open subset $f^{-1}(W_{\mathrm{reg}})$ is smooth.
		\end{Lemma}
		\begin{Proposition}[Enumerativity, \cite{horospherical}]
			Let $X$ be a $G$-variety with finitely many $G$-orbits. Assume that $\mathcal{M}_{0,3}(X,1)$ is irreducible and is equal to the closure of $\mathcal{M}^*_{0,3}(X,1)$. Take $Z_1,Z_2,Z_3$ to be subvarieties of $X$ that intersect the orbit stratification properly and that represent cohomology classes $\Gamma_1,\Gamma_2,\Gamma_3$,
			satisfying $\sum_i \mathrm{codim}_X Z_i = \mathrm{dim} \mathcal{M}_{0,3}(X,1)$.
			Then there is a dense open subset $U \subset G^3$ such that for all $(g_1,g_2,g_3) \in U$, the Gromov-Witten invariant $I_{0,3,1}(\Gamma_1,\Gamma_2,\Gamma_3)$ is equal to the number of isomorphism classes of stable maps with irreducible source of genus 0 and degree 1 in with marked points sent into the translates $g_1Z_1,g_2Z_2, g_3Z_3$ respectively.
		\end{Proposition}
		
		In the calculation, we will use the intersections of restrictions of the Schubert varieties in $\mathrm{Gr}(3,S^2W)$ and $\mathrm{Gr}(2,S^{2,1}W)$ with $Y$ to serve as $Z_i$ above, because in general these varieties intersect the orbit stratification properly. Moreover, notice that a general plane $\mathds{P}^2$ contained in $Y$ also intersects the orbit stratification properly, which can be seen from the parametrization $(a:b:c)\in \mathds{P}^2\mapsto xz\wedge yz\wedge (ax^2+bxy+cy^2)$. 
		
		Notice that the restriction of $\mathcal{U}_2^*$ to this $\mathds{P}^2$ is equal to $\mathcal{O}_{\mathds{P}^2}\oplus \mathcal{O}_{\mathds{P}^2}\oplus \mathcal{O}_{\mathds{P}^2}(1)$ and the restriction of $\mathcal{U}_1^*$ to this $\mathds{P}^2$ is equal to $\mathcal{O}_{\mathds{P}^2}\oplus \mathcal{O}_{\mathds{P}^2}(1)$. As a result, we have the following Lemma:
		\begin{Lemma}\label{chow p2}
			$(c_1^2,[\mathds{P}^2])=1, \ (c_2,[\mathds{P}^2])=(d_2,[\mathds{P}^2])=0$ and $[\mathds{P}^2]=-c_2^2+3c_2d_2-2d_2^2$.
		\end{Lemma}
		
		\begin{Lemma}
			$I_{0,3,1}(c_1,c_1^2,[point])=3, \ I_{0,3,1}(c_1,c_2,[point])=I_{0,3,1}(c_1,d_2,[point])=0$.
		\end{Lemma}
		\begin{proof}
			By the Divisor Axiom for Gromov-Witten invariants, it is equivalent to calculating $I_{0,2,1}(c_1^2,[point])$,\\
			$I_{0,2,1}(c_2,[point]),I_{0,2,1}(d_2,[point])$. By the Enumerativity Proposition above, we only need to calculate the number of lines passing through a general fixed point and incident to $c_1^2,c_2,d_2$, taking into account the number of intersection points of the line with $c_1^2,c_2,d_2$. Such a line must be contained in a unique plane $\mathds{P}^2$ contained in $Y$ by \cref{lineplane}. By \cref{3 planes}, through the general fixed point, there pass 3 planes $\mathds{P}^2$. Now the calculation directly follows from the previous Lemma.
		\end{proof}
		
		\begin{Lemma}\label{c1,c3,l}
			$I_{0,3,1}(c_1,c_3,[\mathds{P}^1])=0.$
		\end{Lemma}
		\begin{proof}
			Without loss of generality, we assume that the line $[\mathds{P}^1]$ is given by $(s:t)\in \mathds{P}^1\mapsto xz\wedge yz\wedge (sx^2-ty^2)$. Through a general point on this $[\mathds{P}^1]$, there pass 3 planes. Of course, all the general points on this $[\mathds{P}^1]$ share a same plane $[\mathds{P}^2]$ passing through them: $(a:b:c)\in \mathds{P}^2\mapsto xz\wedge yz\wedge (ax^2+bxy+cy^2)$. Using \cref{graber}, we only need to concentrate on general lines when counting curves. A general line $l$ intersecting this fixed $\mathds{P}^1$ and incident to $c_3$ must be contained in one of these planes $\Pi$ by \cref{lineplane}. When $c_3$ is in general position, $\Pi$ cannot be the commonly shared plane because $c_3$ will not intersect this fixed plane. We assume this plane $\Pi$ is given by $\Pi_t: \ (a:b:c)\in \mathds{P}^2\mapsto (x-ty)z\wedge (x+ty)(x-ty) \wedge (a(x+ty)^2+b(x+ty)z+cz^2)$. (Here is a little clarification: when $t\neq 0$, the plane $\Pi_t$ intersects $[\mathds{P}^1]$ in one single point $p$ represented by $xz\wedge yz\wedge(x^2-t^2z^2)$ and the three planes passing through $p$ are just $[\mathds{P}^2]$, $\Pi_t$ and $\Pi_{-t}$.)
			
			The Schubert cycle $c_3$ is the subvariety of three dimensional subspaces of $S^2W$ lying in the kernel of a general linear form $f\in S^2W^*$. The intersection of any possible plane $\Pi$ with $c_3$ is empty because for general $f$, there isn't any $t\in \mathds{C}$ such that $\langle f, (x-ty)z\rangle=\langle f,(x+ty)(x-ty)\rangle=0$. As a result, such a line $l$ cannot exist and the Gromov-Witten invariant $I_{0,2,1}(c_3,[\mathds{P}^1])$ is zero. 
		\end{proof}
		
		\begin{Lemma} \label{c_1c_2,l}
			$I_{0,3,1}(c_1,c_1c_2,[\mathds{P}^1])=3.$
		\end{Lemma}
		\begin{proof}
			Recall that to define the Schubert cycle $c_2$, we need to pick two general elements $q_1,q_2\in S^2W^*$ and $c_2$ is the subvariety of three dimensional subspaces such that $q_1,q_2$ are linearly dependent when considered as linear forms on these three dimensional subspaces.
			
			As in the proof of the last Lemma, $[\mathds{P}^1]$ is given by $(s:t)\in \mathds{P}^1\mapsto xz\wedge yz\wedge (sx^2-ty^2)$. A general line $l$ intersecting this fixed $\mathds{P}^1$ and incident to $c_1c_2$ must be contained in a unique plane $\Pi$ of the form $(a:b:c)\in \mathds{P}^2\mapsto (x-ty)z\wedge (x+ty)(x-ty) \wedge (a(x+ty)^2+b(x+ty)z+cz^2)$ by \cref{lineplane}. In order for such a plane $\Pi$ to intersect $c_2$, we definitely need that $$\det\begin{pmatrix}
				\langle q_1, (x-ty)z\rangle & \langle q_1,(x+ty)(x-ty)\rangle\\
				\langle q_2, (x-ty)z\rangle & \langle q_2,(x+ty)(x-ty)\rangle
			\end{pmatrix}=0.$$ When $q_1,q_2$ are general, there can be only three distinct $0\neq t\in \mathds{C}$ making the equality holds, which corresponds to three possible choices of the plane $\Pi$. From the explicit parametrization of such $\Pi$, the intersection of $\Pi$ with $c_1c_2$ is just a point, and the intersection of $\Pi$ with $[\mathds{P}^1]$ is another point. The line $l$ we are looking for is just the line lying in $\Pi$ connecting the two distinct intersection points. As a result, we have $I_{0,2,1}(c_1c_2,[\mathds{P}^1])=3$.
		\end{proof}
		\begin{Lemma}
			$I_{0,3,1}(c_1,c_1^3,[\mathds{P}^1])=8, \ I_{0,3,1}(c_1,c_1d_2,[\mathds{P}^1])=2$.
		\end{Lemma}
		\begin{proof}
			Recall that we have the involution $g_2$ and it maps $c_2$ to $3d_2-c_2$ by \cref{involution} and fixes $[\mathds{P}^1]$, which is the effective generator of the degree 5 Chow group. Hence we must have $I_{0,2,1}(c_1c_2,[\mathds{P}^1])=I_{0,2,1}(c_1(3d_2-c_2),[\mathds{P}^1])$, from which we deduce that $I_{0,2,1}(c_1c_2,[\mathds{P}^1])=2$. Noticing that $c_1^3=4c_1d_2-3c_3$, we deduce that $I_{0,3,1}(c_1,c_1^3,[\mathds{P}^1])=8$.
		\end{proof}
		\begin{Lemma}
			$I_{0,3,1}(c_1,d_2^2,d_2^2)=6$. \label{d_2^2,d_2^2}
		\end{Lemma}
		\begin{proof}
			Recall that $Y\subset \mathrm{Gr}(2,\mathfrak{sl}(W))$ is the subvariety of abelian planes. To define the Schubert cycle $d_2$, we pick a general seven dimensional subspace of $\mathfrak{sl}(W)$ and let $d_2$ be the subvariety of abelian planes contained in this general seven dimensional subspace. For the two $d_2^2$ appearing in the statement of the Lemma, we pick two general 6-dimensional subspaces $U_6$ and $U_6'$ respectively and define $U_4=U_6\cap U_6'$. 
			
			A line $l$ contained in $Y$ is determined by $V_1\subset V_3$, where $V_1$ represents a point in $\mathds{P}(W)\times \mathds{P}(W^*)\subset\mathds{P}(\mathfrak{sl}(W))$ and $V_3$ is contained in the four dimensional subspace $V_4$ determined by $V_1$ (see \cref{plane form}). In order for the line $l$ to intersect the two $d_2^2$, $V_1$ must be contained in both $U_6$ and $U_6'$, thus also in $U_4$. So we require that $V_1$ represents a point in $(\mathds{P}(W)\times \mathds{P}(W^*))\cap\mathds{P}(U_4)$. Since $U_4$ is general, the latter intersection is a 6-point set, because the degree of the Segre embedding $\mathds{P}(W)\times \mathds{P}(W^*)\subset \mathds{P}(W\otimes W^*)$ is 6.
			
			When $V_1$ has been chosen, we have the 4-dimensional space $V_4$ consisting of matrices commuting with $V_1$. Then $U_6\cap V_4$ is an abelian plane containing $V_1$, which represents the intersection point of $d_2^2$ with the plane $\Pi\subset Y$ given by $V_1\subset V_4$. The intersection of another $d_2^2$ with the plane $\Pi$ given by $V_1\subset V_4$ is again one single point represented by $U_6'\cap V_4$. The line $l$ is the line lying in this $\Pi\subset Y$ connecting the two intersection points mentioned above.
		\end{proof}
		
		\begin{Lemma}
			$I_{0,3,1}(c_1,d_2^2,[\mathds{P}^2])=2$.\label{d_2^2,P^2}
		\end{Lemma}
		\begin{proof}
			As in the last Lemma, we assume that the $d_2^2$ is given by a general six dimensional subspace $U_6\subset \mathfrak{sl}(W)$. Without loss of generality, we assume that $[\mathds{P}^2]$ is given by  $V_1=\mathds{C}\begin{pmatrix}
				1&0&0\\
				0&1&0\\
				0&0&-2
			\end{pmatrix}\subset V_4$, and $V_4$ is the space of matrices in $\mathfrak{sl}(W)$ of the form $\begin{pmatrix}
				a&b&0\\
				c&d&0\\
				0&0&-a-d
			\end{pmatrix}$.
			
			The line $l$ incident to both $d_2^2$ and $[\mathds{P}^2]$ will be determined by $V_1'\subset V_3'$. We need that $V_1'\subset V_4$, $V_1'\subset U_6$ and $V_1'$ should represent a point in $\mathds{P}(W)\times \mathds{P}(W^*)\subset\mathds{P}(\mathfrak{sl}(W))$. The last condition is equivalent to saying that for any nonzero matrix $M=\begin{pmatrix}
				a&b&0\\
				c&d&0\\
				0&0&-a-d
			\end{pmatrix}$ in $V_1'\subset V_4$, there exists some $\lambda\in \mathds{C}$ such that $M-\lambda.\mathrm{Id}$ is of rank 1. Excluding some low dimensional possibilities using the fact that $V_1'\subset U_6$ and $U_6$ is general, matrices in $[V_1']\subset \mathds{P}(V_4)$ must satisfy the quadric condition $\det\begin{pmatrix}
				a-(-a-d) & b\\
				c & d-(-a-d)
			\end{pmatrix}=0$, which describes a degree 2 surface in $\mathds{P}(V_4)$. Taking into account that $V_1'\subset U_6$, we see that there can be only two choices of $V_1'$. Once $V_1'$ is fixed, then the point of intersection of $l$ with $[\mathds{P}^2]$ must be represented by the abelian plane $V_1'+V_1$. We conclude as in the proof of the last Lemma.   
		\end{proof}
		\begin{Lemma}
			$I_{0,3,1}(c_1,c_1c_3,[\mathds{P}^2])=1$, $I_{0,3,1}(c_1,[\mathds{P}^2],[\mathds{P}^2])=1$.
		\end{Lemma}
		\begin{proof}
			The proof of this Lemma is similar to that of Lemma \ref{c1,c3,l}. We assume that $[\mathds{P}^2]$ is represented by $(a:b:c)\in \mathds{P}^2\mapsto xz\wedge yz\wedge (ax^2+bxy+cy^2)$. Since any line $l\subset Y$ is contained in a unique plane $\Pi\subset Y$ by \cref{lineplane}, we first take care of planes in $Y$. Using \cref{graber} (counting dimension), we assume the line is general and thus the plane is general. A general plane $\Pi$ intersecting this fixed $\mathds{P}^2$ is given by $(a:b:c)\in \mathds{P}^2\mapsto (x+sy)z\wedge (x+sy)(x+ty) \wedge (a(x+ty)^2+b(x+ty)z+cz^2)$. The condition that this plane $\Pi$ intersects $c_3$ will determine $s,t\in \mathds{C}$ uniquely and with this specific $(s,t)$, $\Pi$ will intersect $c_1c_3$ in only one point. On the other hand, the intersection of this plane with the fixed $\mathds{P}^2$ must be the single point represented by $xz\wedge yz\wedge (x+sy)(x+ty)$. The line $l$ that we are looking for is the line contained in $\Pi$ connecting the two intersection points mentioned above. This shows that $I_{0,3,1}(c_1,c_1c_3,[\mathds{P}^2])=1$. Noticing that $c_1c_3=c_2^2-3c_2d_2+3d_2^2=d_2^2-[\mathds{P}^2]$, we conclude that $I_{0,3,1}(c_1,[\mathds{P}^2],[\mathds{P}^2])=1$.
		\end{proof}
		
		\begin{Lemma}
			$I_{0,3,1}(c_1,c^2_2,[\mathds{P}^2])=4$, $I_{0,3,1}(c_1,c_2^2,d_2^2)=13$.
		\end{Lemma}
		\begin{proof}
			We use the fact that the involution $g_2$ interchanges $c_2$ with $3d_2-c_2$ and fixes $d_2$. Moreover, we know from \cref{chow p2} that $[\mathds{P}^2]=-c_2^2+3c_2d_2-2d_2^2=(c_2-2d_2)(-c_2+d_2)$ and is thus fixed by the involution. $c_2^2$ is send by the involution to $(3d_2-c_2)^2=-c^2_2-2[\mathds{P}^2]+5d_2^2$. Then the equalities $I_{0,3,1}(c_1,c^2_2,[\mathds{P}^2])=I_{0,3,1}(c_1,-c^2_2-2[\mathds{P}^2]+5d_2^2,[\mathds{P}^2])$ and $I_{0,3,1}(c_1,c^2_2,d_2^2)=I_{0,3,1}(c_1,-c^2_2-2[\mathds{P}^2]+5d_2^2,d_2^2)$ allow us to conclude.
		\end{proof}
		
		\begin{Lemma}
			$I_{0,3,1}(c_1,c_2^2,c_2^2)=24$.
		\end{Lemma}
		\begin{proof}
			We first look for lines $l$ incident to the two $c_2^2$ appearing in the statement. Using \cref{graber} (counting dimension), we can assume that this line is a general line in the sense of \cref{terminology}. Such a line $l$ is contained in a unique general plane $\Pi$ of the form $(a:b:c)\in \mathds{P}^2\mapsto l_1l\wedge l_2l\wedge (al_1^2+bl_1l_2+cl_2^2)$, given by an element in $\{(p_1,l)\in \mathds{P}(W^*)\times \mathds{P}(W)|p_1\notin l\}$, where $\langle l_1,l_2\rangle$ is the space of lines passing through $p_1$. In order for such a plane $\Pi$ to be incident to $c_2$, we need to first impose a condition on $l_1l$ and $l_2l$: pick two general linear forms $q_1,q_2\in S^2W^*$, we require that $\det\begin{pmatrix}
				q_1(l_1l)&q_1(l_2l)\\
				q_2(l_1l)&q_2(l_2l)
			\end{pmatrix}=0$. Since we have four $c_2$ in the statement of the Lemma, we impose four conditions on $(\langle l_1,l_2\rangle,l)\in G(2,W)\times \mathds{P}(W)\cong \mathds{P}(W^*)\times \mathds{P}(W)$. 
			
			Instead of directly calculating how many $(\langle l_1,l_2\rangle,l)$ satisfies our conditions, we use the following argument. We let $W^*\rightarrow Q$ be the universal quotient map on $\mathds{P}(W^*)$. We have the map $S^2W^*\subset W^*\otimes W^*\rightarrow Q\boxtimes \mathcal{O}_{\mathds{P}(W)}(1)$ over $\mathds{P}(W^*)\times \mathds{P}(W)$, which gives us global sections $S^2W^*\subset H^0(Q\boxtimes \mathcal{O}_{\mathds{P}(W)}(1))$ generating the vector bundle. The two general linear forms $q_1,q_2\in S^2W^*$ can be thus viewed as global sections of $Q\boxtimes \mathcal{O}_{\mathds{P}(W)}(1)$ and we can use them to define the subvariety representing $c_1(Q\boxtimes \mathcal{O}_{\mathds{P}(W)}(1))$. This subvariety is exactly the locus satisfying the condition we mentioned in the last paragraph. Then the fundamental class of the subvariety satisfying the condition that we impose on $l_1l$ and $l_2l$ from one $c_2$ is equal to the first Chern class of $Q\boxtimes \mathcal{O}_{\mathds{P}(W)}(1)$, which is $h_1+2h_2$, where $h_1,h_2$ are hyperplane classes for $\mathds{P}(W^*),\mathds{P}(W)$ respectively. Noticing $(h_1+2h_2)^4=24$, we conclude that there can be only 24 possible planes incident to $c_2^2$ and $c_2^2$.
			
			Moreover, we see from the parametrization of $\Pi$ that each of the 24 planes intersects both $c_2^2$ in only one point. The line $l$ that we are looking for is just the line lying in one of these 24 planes connecting the two intersection points.  
		\end{proof}
		\begin{Corollary}
			Viewing $I_{0,3,1}(c_1,*,*)$ as a symmetric bilinear form on $A^4(Y,\mathds{C})$, the value of this form is summarized as follows: 
			
			\begin{center}
				\tabcolsep 5pt
				\renewcommand{\arraystretch}{1.4}
				\arrayrulewidth 1pt
				\begin{tabular}{|c|c|c|c|c|}
					\hline $I_{0,3,1}(c_1,*,*)$& $c_2^2$ &$c_2d_2$&$d_2^2$&$[\mathds{P}^2]$\\
					\hline $c_2^2$ &24&18&13&4\\
					\hline $c_2d_2$& 18&13&9&3\\
					\hline $d_2^2$&13&9&6&2\\
					\hline $[\mathds{P}^2]$&4&3&2&1\\
					\hline\end{tabular}
			\end{center}
		\end{Corollary}

		\begin{Lemma}
			$I_{0,3,1}(d_2,d_2,[\mathds{P}^1])=0.$
		\end{Lemma}
		\begin{proof}
			This proof is similar to that of Lemma \ref{d_2^2,P^2}. Assume that this $[\mathds{P}^1]$ is given by $\mathds{C}\begin{pmatrix}
				1&0&0\\
				0&1&0\\
				0&0&-2
			\end{pmatrix}\subset U_3$, where $U_3$ is contained in the four dimensional subspace $U_4$ of matrices in $\mathfrak{sl}(W)$ of the form $\begin{pmatrix}
				a&b&0\\
				c&d&0\\
				0&0&-a-d
			\end{pmatrix}$.
			The line $l$ that we are looking for is in the form $V_1\subset V_3$, such that nonzero matrices in $V_1$ are of rank 1 after additions of scalar matrices. In order for $l$ to be incident to the two $d_2$ and $[\mathds{P}^1]$, we need that $V_1$ is contained in $\mathds{P}(U_3)\cap\mathds{P}(U_6)=\{point\}$, where $U_6$ is a general 6-dimensional subspace of $\mathfrak{sl}(W)$. But a general matrix of the form $\begin{pmatrix}
				a&b&0\\
				c&d&0\\
				0&0&-a-d
			\end{pmatrix}$ isn't of rank 1 after additions of any scalar matrices. This tells us that such $l$ which we are looking for cannot exist.
		\end{proof}
		
		\begin{Lemma}
			$I_{0,3,1}(c_2,d_2,[\mathds{P}^1])=0$.
		\end{Lemma}
		\begin{proof}
			The involution $g_2$ maps $c_2$ to $3d_2-c_2$, so we have $I_{0,3,1}(c_2,c_2,[\mathds{P}^1])=I_{0,3,1}(3d_2-c_2,3d_2-c_2,[\mathds{P}^1])=0$ which allows us to conclude.
		\end{proof}
		\begin{Lemma}
			$I_{0,3,1}(c_2,c_2,[\mathds{P}^1])=0$.
		\end{Lemma}
		\begin{proof}
			This proof is similar to that of Lemma \ref{c_1c_2,l}. We assume that $[\mathds{P}^1]$ is given by $(s:t)\in \mathds{P}^1\mapsto xz\wedge yz\wedge (sx^2-ty^2)$. A general plane $\Pi$ incident to this $[\mathds{P}^1]$ must be in the form $(a:b:c)\in \mathds{P}^2\mapsto (x-ty)z\wedge (x+ty)(x-ty) \wedge (a(x+ty)^2+b(x+ty)z+cz^2)$. In order for such a plane $\Pi$ to intersect one of $c_2$, we need that $\det\begin{pmatrix}
				\langle q_1, (x-ty)z\rangle & \langle q_1,(x+ty)(x-ty)\rangle\\
				\langle q_2, (x-ty)z\rangle & \langle q_2,(x+ty)(x-ty)\rangle
			\end{pmatrix}=0$, where $q_1,q_2\in S^2W^*$. When $q_1,q_2$ are general, there can be only three distinct $0\neq t\in \mathds{C}$ making the equality hold, which corresponds to three possible choices of the plane $\Pi$. In order for such a plane $\Pi$ to intersect another $c_2$, we impose one extra condition on $t\in\mathds{C}$ as above. Consequently, there can be no plane $\Pi$ incident to the two $c_2$ and $[\mathds{P}^1]$. Any line $l$ incident to the two $c_2$ and $[\mathds{P}^1]$ must be contained in a unique $\Pi$ by \cref{lineplane}, which can assumed to be a general plane because we can assume that this line is general using \cref{graber}. We conclude that $I_{0,3,1}(c_2,c_2,[\mathds{P}^1])=0$.
		\end{proof}
		\begin{Corollary}
			\begin{equation*}
				c_1*c_1^2=c_1^3+q(3[Y]), \ c_1*c_2=c_1c_2, \ c_1*d_2=c_1d_2,
			\end{equation*}
			\begin{equation*}
				c_1*c_1^3=c_1^4+q(8c_1), \ c_1*c_1c_2=c_1^2c_2+q(3c_1), \ c_1*c_1d_2=c_1^2d_2+q(2c_1),
			\end{equation*}
			\begin{equation*}
				c_2*c_2=c_2^2, \  c_2*d_2=c_2d_2, \  d_2*d_2=d_2^2,
			\end{equation*}
			\begin{equation*}
				c_1*c_2^2=c_1c_2^2+q(4c_1^2-3c_2-d_2), \ c_1*c_2d_2=c_1c_2d_2+q(3c_1^2-c_2-3d_2), \ c_1*d_2^2=c_1d_2^2+q(2c_1^2-3d_2).
			\end{equation*}
		\end{Corollary}
		\begin{proof}
			For convenience in checking these equalities, we notice that the dual bases of $c_2^2, \ c_2d_2, \ d_2^2$ in $A^2(Y,\mathds{Q})$ are equal to $-c_1^2+3d_2, \ 3c_1^2+2c_2-12d_2, \ -2c_1^2-3c_2+11d_2$.
		\end{proof}
		\begin{thm}
			The $\mathds{C}$-coefficient quantum cohomology ring $(QH^*(Y),*)$ is presented as the quotient of graded polynomial ring $\mathds{C}[c_1,c_2,d_2,q]$ by the homogeneous ideal generated by the following homogeneous elements:
			$$c_1*c_1*c_1*c_1-11q*c_1=-3c_2*c_2+9c_2*d_2+3d_2*d_2$$
			$$c_1*c_1*c_2-3q*c_1=c_2*d_2+3d_2*d_2$$
			$$c_1*c_1*d_2-2q*c_1=3d_2*d_2$$
			$$c_1*c_2*c_2-q*(4c_1*c_1-3c_2-d_2)=\frac{14}{9
			}(c_1*c_2*d_2-q*(3c_1*c_1-c_2-3d_2))
			$$
			$$c_1*c_2*d_2-q*(3c_1*c_1-c_2-3d_2)=\frac{3}{2}(c_1*d_2*d_2-q*(2c_1*c_1-3d_2))$$
			In particular, the small quantum cohomology ring is reduced.\label{quantum thm}
		\end{thm}
		\begin{proof}
			This directly follows from \cref{chow presented}, \cref{Tian} and the above calculations of quantum products.
		\end{proof}
		As a Corollary of the previous Theorem and \cref{bigsemisimple}, we have
		\begin{Corollary}
			The big quantum cohomology ring of $Y$ is generically semisimple.\label{semisimple}
		\end{Corollary}
		Now we can finish the Quantum Chevalley formula using the above description of the quantum cohomology ring:
		\begin{equation}
			c_1*[\mathds{P}^1]=[point]+q(\frac{2}{3}c_1^3-\frac{5}{3}c_1d_2)+q^2I_{0,3,2}(c_1,[\mathds{P}^1],[point])[Y]\label{unkown1}
		\end{equation}
		\begin{equation}
			c_1*[point]=q(-3c_2^2+9c_2d_2-6d_2^2)+q^2I_{0,3,2}(c_1,[point],[\mathds{P}^1])c_1\label{unkown2}
		\end{equation}
		The only quantity to determine is $I_{0,3,2}(c_1,[\mathds{P}^1],[point])=I_{0,3,2}(c_1,[point],[\mathds{P}^1])$. 
		\begin{Lemma}
			$I_{0,3,2}(c_1,[\mathds{P}^1],[point])=2$.\label{d=2GW}
		\end{Lemma}
		\begin{proof}
			Instead of directly calculating this quantity, we multiply (quantum product) \cref{unkown1} by $c_1$, and substitute the term $c_1*[point]$ using \cref{unkown2}. Notice that $$[\mathds{P}^1]=\frac{1}{14}c_1c_2^2=\frac{1}{14}(c_1*c_2*c_2-q(4c_1*c_1-3c_2-d_2)).$$ Noticing also that $c_1^3=c_1*c_1*c_1-8q[Y], \ c_1*d_2=c_1d_2, \ c_2*c_2=c_2^2, \ c_2*d_2=c_2d_2, \ d_2*d_2=d_2^2$, we have a relation $R$ in the quantum cohomology ring between the generators $c_1,c_2,d_2,q$ with $$R\coloneqq R_1+(2-2I_{0,3,2}(c_1,[\mathds{P}^1],[point]))q^2c_1,$$ where $R_1$ is an explicit polynomial not involving the unknown $I_{0,3,2}(c_1,[\mathds{P}^1],[point])$. One can check that $R'\coloneqq R_1-2q^2c_1$ lies in the ideal described in \cref{quantum thm}. If $I_{0,3,2}(c_1,[\mathds{P}^1],[point])\neq 2$, then this together with the fact that $R-R'=0\in QH^*(Y)$ would imply that $q^2c_1=0$ in the quantum cohomology ring, which is impossible since we have the vector space isomorphism: $QH^*(Y,\mathds{C})\cong H^*(Y,\mathds{C})\otimes _{\mathds{C}}\mathds{C}[q]$.
		\end{proof}
		\begin{Corollary}[Quantum Chevalley formula]
			\begin{alignat*}{4}
				&&c_1*e_1=f_1+f_2, \quad\quad\quad\quad\quad&& c_1*e_2=2f_1+f_2+2f_3+3q[Y], \quad\quad&& c_1*e_3=f_2+f_3,\quad\quad\\
				&&c_1*f_1=h_1+2h_2+qc_1,\quad\quad && c_1*f_2=h_1+h_2+h_3, \quad\quad\quad\quad&& c_1*f_3=h_3+2h_2+qc_1,\\
				&&c_1*h_1=2[line]+qe_3, \quad\quad\quad&& c_1*h_2=[line]+qe_2,\quad\quad \quad\quad&& c_1*h_3=2[line]+qe_1,
			\end{alignat*}
			$$c_1*[\mathds{P}^1]=[point]+q(f_1+f_3)+2q^2[Y], \ \ \ \ \ \ \ \ \ \ \ \ \ \ \ \  c_1*[point]=q(3h_2)+2q^2c_1.$$
		\end{Corollary}
		
		%\begin{remark}
		%  Recall that $\mathrm{Hilb}^3(\mathds{P}^2)$ is the blow up of $Y$ along the orbit $\mathcal{O}_2$. The quantum cohomology of $\mathrm{Hilb}^3(\mathds{P}^2)$ has not been explicitly described. One may hope that the explicit formula for $Y$ can help to reveal some information of the quantum cohomology of $\mathrm{Hilb}^3(\mathds{P}^2)$.
		%\end{remark}

		\section{Derived Category and Quantum Cohomology}
		In this section, we verify the conjectured link between the derived category and the quantum cohomology for our variety $Y$.
		By \cite{Derquiver}, we have the following full exceptional collection for $\mathrm{D}^b(Y)$:
		\begin{equation} \left\langle 
			\mathfrak{sl}(\mathcal{U}_1), \
			\mathcal{O}_Y, \mathcal{U}_2^*, \mathcal{U}_1^*, \mathcal{U}_2(1), \
			\mathcal{O}_Y(1), \mathcal{U}_2^*(1), \mathcal{U}_1^*(1), \mathcal{U}_2(2),  \
			\mathcal{O}_Y(2), \mathcal{U}_2^*(2), \mathcal{U}_1^*(2), \mathcal{U}_2(3)\right\rangle.
			\label{1 exc}
		\end{equation}
		After mutating the object $\mathfrak{sl}(\mathcal{U}_1)$ across $\langle\mathcal{O}_Y, \mathcal{U}_2^*, \mathcal{U}_1^*, \mathcal{U}_2(1)\rangle$, we get a full Lefschetz collection: 
		\begin{equation} \left\langle 
			\mathcal{O}_Y, \mathcal{U}_2^*, \mathcal{U}_1^*, \mathcal{U}_2(1), \mathfrak{sl}(\mathcal{U}_1)(1), \
			\mathcal{O}_Y(1), \mathcal{U}_2^*(1), \mathcal{U}_1^*(1), \mathcal{U}_2(2),  \
			\mathcal{O}_Y(2), \mathcal{U}_2^*(2), \mathcal{U}_1^*(2), \mathcal{U}_2(3)\right\rangle.
			\label{2 exc}
		\end{equation}
		where the starting block of the Lefschetz collection is $\langle\mathcal{O}_Y, \mathcal{U}_2^*, \mathcal{U}_1^*, \mathcal{U}_2(1), \mathfrak{sl}(\mathcal{U}_1)(1) \rangle$.
		
		The famous Dubrovin's Conjecture says that for a smooth Fano variety $X$, the big quantum cohomology of $X$ is generically semisimple if and only if $\mathrm{D}^b(X)$ admits a full exceptional collection. The above full exceptional sequences and \cref{semisimple} now confirm:
		\begin{Proposition}
			Dubrovin's Conjecture holds for $Y$.\label{final prop}
		\end{Proposition}
		There is a refined version \cite{Kuznetsov_Smirnov_2021} of Dubrovin's Conjecture by Kuznetsov and Smirnov. The refined version predicts an $\mathrm{Aut}(Y)$-invariant full Lefschetz collection, with sizes of blocks being $(5,4,4)$. 
		To get an $\mathrm{Aut}(Y)$-invariant full Lefschetz collection, we do the left mutation of $\mathcal{U}_2(1)$ across $\mathcal{U}_1^*$ (and also the left mutation of $\mathcal{U}_2(2)$ across $\mathcal{U}_1^*(1)$ and of $\mathcal{U}_2(3)$ across $\mathcal{U}_1^*(2)$).
		
		By \cite{Derquiver}, we have $\mathrm{RHom}^{\bullet}(\mathcal{U}_1^*,\mathcal{U}_2(1))=\mathrm{RHom}^{\bullet}(\mathcal{U}_1^*(-1),\mathcal{U}_2)=\mathrm{RHom}^{\bullet}(\mathcal{U}_1,\mathcal{U}_2)=W^*[0]$. And the resulting map $\mathcal{U}_1^*\otimes \mathrm{Hom}^{\bullet}(\mathcal{U}_1^*,\mathcal{U}_2(1))\rightarrow \mathcal{U}_2(1)$ is actually surjective: we only need to check the surjectivity fiberwise using the explicit interpretation of $\mathrm{RHom}^{\bullet}(\mathcal{U}_1,\mathcal{U}_2)=W^*[0]$ in \cref{univrep}. The reader can see how to check the surjectivity at the general fibers in the proof of the following Proposition as an example.
		
		We let $S$ denote the kernel of the surjective map $\mathcal{U}_1^*\otimes \mathrm{Hom}^{\bullet}(\mathcal{U}_1^*,\mathcal{U}_2(1))\rightarrow \mathcal{U}_2(1)$. Recall that we have an isomorphism $f: Y\cong Y^*$ and after fixing a nondegenerate bilinear form on $W$, we get an involution $g_0:Y\rightarrow Y$ in \cref{automorphismsection}.
		
		\begin{Proposition}
			$g_0^*\mathcal{U}_1\cong \mathcal{U}_1, \ g_0^*\mathcal{U}_2\cong S^*$.
		\end{Proposition}
		\begin{proof}
			By definition of $S$, we have the short exact sequence $$0\rightarrow S\rightarrow \mathcal{U}_1^*\otimes W^*\rightarrow \mathcal{U}_2(1)\rightarrow0.$$
			Fixing an isomorphism of $W$ with $W^*$ as in the definition of the involution $g_0$, we get the map $S\rightarrow \mathcal{U}_1^*\otimes W$. I claim that the fiber map of $S\rightarrow \mathcal{U}_1^*\otimes W$ is a stable representation of the 3-Kronecker quiver. To show this, we only need to check this fact after twisting with $\mathcal{O}(-1)$, i.e. we consider the map $S(-1)\rightarrow\mathcal{U}_1^*(-1)\otimes W^*\cong\mathcal{U}_1\otimes W^*$ (Recall that $\det(\mathcal{U}_1)=\mathcal{O}_Y(-1)$). For example, over the point $[R]=\begin{pmatrix}
				x & y&0\\
				0&y&z
			\end{pmatrix}\in \mathcal{O}_6$, the map $\mathcal{U}_1\otimes W^*\rightarrow \mathcal{U}_2$ is in the form 
			\begin{align*}
				\begin{array}{lll}
					f_1\otimes X\rightarrow e_1, & f_1\otimes Y\rightarrow e_2,& f_1\otimes Z\rightarrow 0,\\
					f_2\otimes X\rightarrow 0, & f_2\otimes Y\rightarrow e_2,&f_2\otimes Z\rightarrow e_3,
				\end{array}
			\end{align*} in the fiber over $[R]$, for some basis $\{f_1,f_2\}$ of $\mathcal{U}_1|_{[R]}$ and for some basis $\{e_1,e_2,e_3\}$ of $\mathcal{U}_2|_{[R]}$. Consequently, we can take the basis of $S(-1)|_{[R]}$ to be $r_1=f_1\otimes Z,  \ r_2=f_1\otimes Y-f_2\otimes Y, \ r_3=f_2\otimes X$ and the map $S(-1)|_{[R]}\rightarrow \mathcal{U}_1|_{[R]}\otimes W^*$ is described by the matrix $\begin{pmatrix}
				Z&Y&0\\
				0&-Y&X
			\end{pmatrix}$ with respect to the basis elements $r_1,r_2,r_3,f_1,f_2$. This matrix represents a stable representation. 
			
			Consequently $\mathcal{U}_1\rightarrow S^*\otimes W^*$ is also a family of stable representations of the 3-Kronecker quiver. By the universal property in \cref{universal} (together with the chosen isomorphism $W\cong W^*$ given by the nondegenerate bilinear form), there exists a unique map $g':Y\rightarrow Y$, such that $g'^*\mathcal{U}_1\cong \mathcal{U}_1\otimes\mathcal{O}(k)$ and $g'^*\mathcal{U}_2\cong S^*\otimes \mathcal{O}(k)$ for some $k\in \mathds{Z}$. Obviously, $g'$ is an automorphism and must preserve the ample generator $\mathcal{O}(1)$ of the Picard group, which implies $\det(\mathcal{U}_1)=\det(g'^*\mathcal{U}_1)=\det(\mathcal{U}_1\otimes \mathcal{O}(k))$ and thus $k=0$. On the other hand, one can get to know the explicit action of $g'$ on the general point $[R]$ as in the last paragraph and conclude that $g'=g_0$.
		\end{proof}
		\begin{Corollary}
			We have an $\mathrm{Aut}(Y)$-invariant full Lefschetz collection for $\mathrm{D}^b(Y)$:
			\begin{equation} \left\langle 
				\mathcal{O}_Y, \mathcal{U}_2^*, S,  \mathcal{U}_1^*, \mathfrak{sl}(\mathcal{U}_1)(1), \
				\mathcal{O}_Y(1), \mathcal{U}_2^*(1),S(1), \mathcal{U}_1^*(1),   \
				\mathcal{O}_Y(2), \mathcal{U}_2^*(2), S(2),\mathcal{U}_1^*(2)\right\rangle,
				\label{3 exc}
			\end{equation}and \cite[Conjecture 1.3]{Kuznetsov_Smirnov_2021} holds for $Y$.
		\end{Corollary}
		\begin{proof}
			Recall that the index of $Y$ is 3 and $\mathrm{Aut}(Y)\cong PGL(W)\rtimes\langle\mathrm{Id}_Y,g_0\rangle$ by \cref{Aut}. The group $PGL(W)$ fixes all the exceptional objects, and the involution $g_0$ would interchange $\mathcal{U}_2^*$ with $S$ (also interchanging $\mathcal{U}_2^*(1)$ with $S(1)$, and $\mathcal{U}_2^*(2)$ with $S(2)$). By calculation, we know that $c_1=0$ defines a length 1 closed subscheme of the length $13$ scheme $\mathrm{Spec} \ QH^*(Y)|_{q=1}$ and thus the length of $\mathrm{QS}^{\times}_Y$ is 12. The residual category of the rectangular part is equal to $\langle \mathfrak{sl}(\mathcal{U}_1)\rangle$ because (\ref{1 exc}) is a full exceptional sequence.
		\end{proof}
		
		We finish with a diagram that encodes all the eigenvalues of the endomorphism of $QH^*(Y)|_{q=1}$ given by multiplication by $c_1$ (drawn in the complex plane). The blue round points are the eigenvalues with multiplicities 1, and the red square points are the eigenvalues with multiplicities 2.
		\begin{center}
			\begin{tikzpicture}
				\begin{axis}[
					title={Eigenvalues of the $c_1$ multiplication},
					xlabel={Real Axis},
					ylabel={Imaginary Axis},
					xmin=-3, xmax=5,
					ymin=-4, ymax=4,
					xtick={-2,-1,0,1,2,3,4},
					ytick={-3,-2,-1,0,1,2,3},
					%legend pos=north west,
					%ymajorgrids=true,
					%grid style=dashed,
					]
					
					\addplot+[
					only marks,
					color=blue,
					mark=*,
					]
					coordinates {
						(0, 0)
						
						(-1.810645079075508, 0)
						(3.446424449092975, 0)
						(-1.723212224546488, -2.984691125138305)
						(-1.723212224546488, 2.984691125138305)
						(0.9053225395377538, -1.568064635716674)
						(0.9053225395377538, 1.568064635716674)
					};
					\addplot+[
					only marks,
					color=red,
					mark=square*,]
					coordinates{(-1, 0)
						(-1, 0)
						(0.50000000000000000, -0.866025403784439)
						(0.50000000000000000, -0.866025403784439)
						(0.50000000000000000, 0.866025403784439)
						(0.50000000000000000, 0.866025403784439)
					};
				\end{axis}
			\end{tikzpicture}
		\end{center}

		\bibliographystyle{alpha}
		\bibliography{Quantum}

\newcommand{\etalchar}[1]{$^{#1}$}
\begin{thebibliography}{ABMBB02}

\bibitem[ABMBB02]{BBbook}
Białynicki-Birula Andrzej, Carell~James B., McGovern~William Monty, and
  Andrzej Białynicki-Birula.
\newblock {\em Algebraic quotients ; Torus actions and cohomology ; The adjoint
  representation and the adjoint action / A. Białynicki-Birula, J.B. Carrell,
  W.M. McGovern}.
\newblock Encyclopaedia of mathematical sciences Invariant theory and algebraic
  transformation groups. Springer, Berlin New York, 2002.

\bibitem[BB73]{BBdecomposition}
A.~Bialynicki-Birula.
\newblock Some theorems on actions of algebraic groups.
\newblock {\em Annals of Mathematics}, 98(3):480--497, 1973.

\bibitem[BBF{\etalchar{+}}23]{belmans2023rigidity}
Pieter Belmans, Ana-Maria Brecan, Hans Franzen, Gianni Petrella, and Markus
  Reineke.
\newblock Rigidity and schofield's partial tilting conjecture for quiver
  moduli, 2023.

\bibitem[BBFR23]{belmans2023vector}
Pieter Belmans, Ana-Maria Brecan, Hans Franzen, and Markus Reineke.
\newblock Vector fields and admissible embeddings for quiver moduli, 2023.

\bibitem[BF23]{belmans2023chow}
Pieter Belmans and Hans Franzen.
\newblock On chow rings of quiver moduli, 2023.

\bibitem[BOR20]{autHilb}
Pieter Belmans, Georg Oberdieck, and J{\o}rgen~Vold Rennemo.
\newblock Automorphisms of {Hilbert} schemes of points on surfaces.
\newblock {\em Trans. Am. Math. Soc.}, 373(9):6139--6156, 2020.

\bibitem[Dre87]{drezet}
J.-M. Drezet.
\newblock Fibr{\'e}s exceptionnels et vari{\'e}t{\'e}s de modules de faisceaux
  semi- stables sur {{\({\mathbb{P}}_ 2({\mathbb{C}})\)}}. (exceptional bundles
  and moduli varieties of semi-stable sheaves on {{\({\mathbb{P}}_
  2({\mathbb{C}}))\)}}.
\newblock {\em J. Reine Angew. Math.}, 380:14--58, 1987.

\bibitem[Dub98]{dubrovin}
Boris Dubrovin.
\newblock Geometry and analytic theory of {Frobenius} manifolds.
\newblock {\em Doc. Math.}, Extra Vol.:315--326, 1998.

\bibitem[EPS87]{EPS}
Geir Ellingsrud, Ragni Piene, and Stein~Arild Str{\o}mme.
\newblock On the variety of nets of quadrics defining twisted cubics.
\newblock Space curves, {Proc}. {Conf}., {Rocca} di {Papa}/{Italy} 1985,
  {Lect}. {Notes} {Math}. 1266, 84-96 (1987)., 1987.

\bibitem[ES95]{ES}
Geir Ellingsrud and Stein~Arild Str{\o}mme.
\newblock The number of twisted cubic curves on the general quintic threefold.
\newblock {\em Math. Scand.}, 76(1):5--34, 1995.

\bibitem[FP97]{FP97}
W.~Fulton and R.~Pandharipande.
\newblock Notes on stable maps and quantum cohomology.
\newblock In {\em Algebraic geometry. Proceedings of the Summer Research
  Institute, Santa Cruz, CA, USA, July 9--29, 1995}, pages 45--96. Providence,
  RI: American Mathematical Society, 1997.

\bibitem[Fra15]{Franzen}
H.~Franzen.
\newblock Chow rings of fine quiver moduli are tautologically presented.
\newblock {\em Math. Z.}, 279(3-4):1197--1223, 2015.

\bibitem[GPPS22]{horospherical}
R.~Gonzales, C.~Pech, N.~Perrin, and A.~Samokhin.
\newblock Geometry of horospherical varieties of {Picard} rank one.
\newblock {\em Int. Math. Res. Not.}, 2022(12):8916--9012, 2022.

\bibitem[Gra01]{Graber}
Tom Graber.
\newblock Enumerative geometry of hyperelliptic plane curves.
\newblock {\em J. Algebr. Geom.}, 10(4):725--755, 2001.

\bibitem[Hir54]{Hirzebruch}
Friedrich Hirzebruch.
\newblock Some problems on differentiable and complex manifolds.
\newblock {\em Annals of Mathematics}, 60(2):213--236, 1954.

\bibitem[IM05]{IlievManivel}
A.~Iliev and L.~Manivel.
\newblock Severi varieties and their varieties of reductions.
\newblock {\em Journal für die reine und angewandte Mathematik},
  2005(585):93--139, 2005.

\bibitem[KS21]{Kuznetsov_Smirnov_2021}
Alexander Kuznetsov and Maxim Smirnov.
\newblock Residual categories for (co)adjoint grassmannians in classical types.
\newblock {\em Compositio Mathematica}, 157(6):1172–1206, 2021.

\bibitem[KW95]{Kingwalter}
Alastair~D. King and Charles~H. Walter.
\newblock On {C}how rings of fine moduli spaces of modules.
\newblock {\em J. Reine Angew. Math.}, 461:179--187, 1995.

\bibitem[Lee04]{YP-Lee}
Y.-P. Lee.
\newblock Quantum {{\(K\)}}-theory. {I}: {Foundations}.
\newblock {\em Duke Math. J.}, 121(3):389--424, 2004.

\bibitem[MM24]{Derquiver}
Svetlana MAKAROVA and Junyu MENG.
\newblock Full exceptional sequence for a fine quiver moduli space.
\newblock {\em upcoming preprint}, 2024.

\bibitem[Pet24]{peternell2024compactificationscncomplexprojective}
Thomas Peternell.
\newblock Compactifications of $\mathds{C}^n$ and the complex projective space,
  2024.

\bibitem[pLQZ01]{Hilb^3}
Wei ping Li, Zhenbo Qin, and Qi~Zhang.
\newblock On the geometry of the hilbert schemes of points in the projective
  plane, 2001.

\bibitem[ST97]{Tian}
Bernd Siebert and Gang Tian.
\newblock On quantum cohomology rings of {Fano} manifolds and a formula of
  {Vafa} and {Intriligator}.
\newblock {\em Asian J. Math.}, 1(4):679--695, 1997.

\bibitem[{Sta}18]{stacks-project}
The {Stacks Project Authors}.
\newblock \textit{Stacks Project}.
\newblock \url{https://stacks.math.columbia.edu}, 2018.

\end{thebibliography}
		\texttt{Junyu.Meng@math.univ-toulouse.fr} \\
		Université Paul Sabatier, Institut de Mathématiques de Toulouse, 118, Route de Narbonne, F-31062 Toulouse Cedex 9, France
	\end{document}